\documentclass[10pt,reqno]{amsart}
\usepackage{amssymb}
\usepackage{amsmath, amssymb, amsthm}
\usepackage{mathrsfs}
\numberwithin{equation}{section}
\usepackage{color}
\newtheorem{theorem}{Theorem}[section]

\newtheorem{lemma}[theorem]{Lemma}
\newtheorem{remark}{Remark}[section]

\newtheorem{prop}{Proposition}[section]
\newcommand{\dif}{\mathrm{d}}

\newcommand{\LIM}[2]{\lim\limits_{{#1}\rightarrow{#2}}}

\begin{document}
\title[Stability for relaxed CNS]{Asymptotic stability of composite waves of viscous shock and rarefaction  for relaxed compressible Navier-Stokes equations}
\author{Renyong Guan and Yuxi Hu}
 \thanks{\noindent  Renyong Guan,   Department of Mathematics, China University of Mining and Technology, Beijing, 100083, P.R. China, renyguan@163.com\\
\indent  Yuxi Hu, Department of Mathematics, China University of Mining and Technology, Beijing, 100083, P.R. China, yxhu86@163.com\\
 }
\begin{abstract}
The time asymptotic stability for one-dimensional relaxed compressible Navier-Stokes equations is studied.
We show that the composite waves of viscous shock and rarefaction are asymptotically nonlinear stable with  both small wave strength and small initial perturbations. Moreover,  as the relaxation parameter goes to zero,  the solutions of relaxed system are shown to converge globally in time to that of classical system. The methods are based on relative entropy, the $a$-contraction with shifts theory and basic energy estimates.
 \\[2em]

\end{abstract}
\maketitle
\section{Introduction}
In this paper, we study the system of one-dimensional isentropic compressible Navier-Stokes equations with Maxwell's constitutive relations. The system can be described as follows:
\begin{align}\label{1.1}
\begin{cases}
\rho_t+(\rho u)_x=0,\\
(\rho u)_t+(\rho u^2)_x+p(\rho)_x=\Pi_x,\\
\tilde \tau  (\rho) (\Pi_t+ u \Pi_x)+\Pi=\mu u_x,
\end{cases}
\end{align}
where $(t, x)\in (0, +\infty)\times \mathbb R$. Here, $\rho$, $u$, $\Pi$ denote fluid density, velocity and stress, respectively. $\mu>0$ is the viscosity constant.
The pressure $p$ is assumed to satisfy the usual $\gamma$-law, $p(\rho)=A \rho^\gamma$ where $\gamma>1$ denotes the adiabatic index and $A$ is any positive constant.
Without loss of generality, we assume 
$A=1$ in the sequel.

The constitutive relation $\eqref{1.1}_3$  was proposed by Maxwell in \cite{MAX}, in order to describe the relation of stress tensor and velocity gradient for a non-simple fluid. The relaxation parameter $\tilde\tau=\tilde \tau(\rho)$ describe the time lag in response of the stress tensor to velocity gradient. We note that even in simple fluid, water for example, the {\it {time lag}} exists, but it is very small ranging from 1 ps to 1 ns, see \cite{GM, FS}. However, Pelton et al. \cite{MP} showed that such a {\it time lag} cannot be neglected, even for simple fluids, in the experiments of high-frequency (20GHZ) vibration of nano-scale mechanical devices immersed in water-glycerol mixtures. It was shown that, see also \cite{DJE}, equation $\eqref{eq1}_3$ provides a general formalism with which to characterize the fluid-structure interaction of nano-scale mechanical devies vibrating in simple fluids.

For simplicity, we assume $\tilde \tau(\rho)=\tau \rho$ with $\tau$ being a  positive constant. Therefore, the equation $\eqref{1.1}_3$ reduces to
\begin{align}\label{1.2}
\tau  \rho (\Pi_t+ u \Pi_x)+\Pi=\mu u_x.
\end{align}
The constitutive equation \eqref{1.2} was proposed firstly by Freist\"uhler \cite{FRE1, FRE2} from the point of view of conservation laws. Indeed, under the above assumption, the equation \eqref{1.2} is conservative by use of mass equation $\eqref{1.1}_1$ and thus weak solutions are easily to defined. For convenience, we rewrite the system \eqref{1.1} with $\tilde \tau(\rho)=\tau \rho$ in Lagrangian coordinates as follows:
\begin{equation}\label{eq1}
\begin{cases}
v_t-u_x=0,\\
u_t+p_x=\Pi_x,\\
\tau \Pi_t+v\Pi=\mu u_x,
\end{cases}
\end{equation}
where $v=\frac{1}{\rho}$ denotes the specific volume per unit mass.

We are interested in the Cauchy problem to system \eqref{eq1}   for the functions
\begin{align*}
(v, u, \Pi): [0, +\infty) \times \mathbb R \rightarrow (0, \infty)\times \mathbb R \times \mathbb R
\end{align*}
with initial conditions
\begin{align} \label{eq2}
(v, u, \Pi)(0,x)=(v_0, u_0, \Pi_0)(x)\rightarrow(v_{\pm},u_{\pm},0) \quad (x\rightarrow\pm\infty).
\end{align}

The large-time behavior of solutions to system \eqref{eq1}-\eqref{eq2} is closely
related to the Riemann problem of the associated $p$-system
\begin{equation}\label{eq3}
\begin{cases}
v_t-u_x=0,\\
u_t+p(v)_x=0,
\end{cases}
\end{equation}
with the Riemann initial data:
\begin{equation}\label{eq4}
(v,u)(t=0,x)=
\begin{cases}
(v_-,u_-),\quad x<0,\\
(v_+,u_+),\quad x>0.
\end{cases}
\end{equation}

If $\tau=0$, the system \eqref{eq1} reduces to classical compressible isentropic Navier-Stokes equations:
\begin{equation}\label{1}
\begin{cases}
v_t-u_x=0,\\
u_t+p(v)_x=\left(\mu\frac{u_x}{v}\right)_x.
\end{cases}
\end{equation}

For system \eqref{1}, the asymptotic behavior of solutions with various initial data have been widely studied, see \cite{HH, HLM, HM, WY, MNS, MNR1, MNR2}. For shock profile initial data, Matsumura and Nishihara \cite{MNS}, Goodman\cite{GD} first established the stability of the traveling wave under the condition that the initial disturbance is suitably small and of zero constant component by using anti-derivative method. Using energy methods, Matsumura and Nishihara \cite{MNR1, MNR2}  established the stability of rarefaction waves. Note that the standard anti-derivative methods, used to study the stability of viscous shocks, are not compatible to the energy method used for the stability of rarefactions. Recently, M.-J. Kang, A. F. Vasseur, Y. Wang  \cite{WY} showed that the composite waves of viscous shock and rarefaction is still stable by using the method of relative entropy and the $a$-contraction with shifts theory, which was first initiated by Bresch and Desjardins in \cite{BD} and  was developed  by  \cite{KV1, KV2}. The aim of this paper is to extend this result to the case of relaxed compressible Navier-Stokes equations \eqref{eq1}.

One should note that it is not obvious that the results which hold for the classical systems also hold for the relaxed system. Indeed, and for example, Hu and Wang \cite{HW}  and Hu, Racke and Wang \cite{HRW} showed that, solutions may blow up in finite time with some large initial data for the relaxed system; while ,for classical compressible Navier-Stokes system, solutions were shown to  exist globally with arbitrary large initial data (away from vacuum) , see \cite{Kanel1968}.  A similar qualitative change was observed before for certain thermoelastic systems, where the non-relaxed system is exponentially stable, while the relaxed one is not, see Quintanilla and Racke resp. Fern\'andez Sare and Mu\~noz Rivera \cite{QuRa011,FeMu012} for plates, and Fern\'andez Sare and Racke \cite{FeRa009} for Timoshenko beams.

For the case of relaxed compressible Navier-Stoeks equations, the time-asymptotic stability of single viscous shock wave or composite of two rarefaction waves have already been studied. For $\tilde \tau(\rho)$ being a constant,  Hu-Wang \cite{ZWH} and Hu-Wang \cite{XFH} established the linear stability of the viscous shock wave and nonlinear  stability of rarefaction waves, respectively. For $\tilde \tau=\tau \rho$ , by checking Majda's condition on the Lopatinski determinant and Zumbrun's Evans function condition \cite{MZ, MZ2, PZ},  Freist\"uhler \cite{FRE1} get the nonlinear stability of the viscous shock waves for system \eqref{1.1} with shock profile initial data.

Motivated by the methods used in \cite{WY}, we studied the time asymptotic stability for composite waves of the superposition  of a viscous shock and a rarefaction to system \eqref{eq1}-\eqref{eq2}. Note that the B-D  entropy used in \cite{WY}, which is essential to prove the  $a$-contraction property, is not available for the studied system \eqref{eq1}. Meanwhile, the dissipation structure of relaxed system \eqref{eq1} is much weaker than that of classical system \eqref{eq3}, thus energy estimates may have new challenges.  Therefore, new methods  and ideas should be developed to overcome such difficulties. Here are our strategies. Instead of using B-D entropy,  by using of the special  hyperbolic structure of the relaxed system and the relative entropy quantities with weight function and shifts, we first established the  $L^\infty_tL^2_x$ estimates of $(v-\widetilde v, u-\widetilde u, S-\widetilde S)$ and weighted $L^2_tL^2_x$ estimates of $(v-\widetilde v, S-\widetilde S)$ (see Lemma \ref{Le5}). Unlike that in \cite{WY}, the dissipation estimates of  the derivative of $(v-\widetilde v)$ can not obtained for system \eqref{eq1} due to the essentially different structure of the two system. Indeed, the change of variables based on B-D entropy is not available for the relaxed system and thus the $a$-contraction property could  be violated. However, we can recover the estimates of derivative of $(v-\widetilde v)$ by carefully doing higher energy estimates. We remark that, compared to the estimates given in \cite{WY}, the $H^2$ estimates  of the solutions are needed to close the energy.

Our main theorem are stated as follows:

\begin{theorem}\label{th1}
Let the relaxation parameter $\tau$ satisfy
\[\tau\leq \min\{\inf\limits_{\widetilde{v}^S\in[v_m,v_+]}\frac{\mu}{|\sigma^2+p^{\prime}(\widetilde{v}^S)|}, 1\}.\]
For a given constant state $(v_+,u_+)\in \mathbb{R}_+\times \mathbb{R}$,
there exist constants $\delta_0,\varepsilon_0>0$ such that the following holds true.\\
For any $(v_m,u_m)\in S_2(v_+,u_+)$ and $(v_-,u_-)\in R_1(v_m,u_m)$ such that
\[|v_+-v_m|+|v_m-v_-|\leq\delta_0.\]
Denote $(v^r,u^r)(\frac{x}{t})$ the 1-rarefaction solution to \eqref{eq3} with end states $(v_-,u_-)$ and $(v_m,u_m)$, and $(\widetilde{v}^S,\widetilde{u}^S,\widetilde{S}^S)(x-\sigma t)$ the 2-viscous shock solution (defined in \eqref{eq25}) with end states $(v_m,u_m,0)$ and $(v_+,u_+,0)$.
Let $(v_0,u_0,S_0)$ be any initial data such that
\begin{equation}\label{eq36}
\sum\limits_{\pm}\left(\|(v_0-v_{\pm},u_0-u_{\pm})\|_{L^2(\mathbb{R}_{\pm})}\right)
+\|((v_0)_x,(u_0)_x)\|_{H^1(\mathbb{R})}+\sqrt{\tau}\|\Pi_0\|_{H^2(\mathbb{R})}<\varepsilon_0,
\end{equation}
where $\mathbb{R}_+:=-\mathbb{R}_-=(0,+\infty)$.\\
Then, the initial value problem \eqref{eq1}-\eqref{eq2} has a unique global-in-time solution (v,u,S). Moreover, there exist an absolutely continuous shift $X(t)$ (defined in \eqref{eq35}) such that
\begin{equation}\label{eq37}
\begin{aligned}
&v(t,x)-\left(v^r(\frac{x}{t})+\widetilde{v}^S(x-\sigma t-X(t))-v_m\right)\in C(0,+\infty;H^2(\mathbb{R})),\\
&u(t,x)-\left(u^r(\frac{x}{t})+\widetilde{u}^S(x-\sigma t-X(t))-u_m\right)\in C(0,+\infty;H^2(\mathbb{R})),\\
&\Pi(t,x)-\left(\widetilde{\Pi}^S(x-\sigma t-X(t))+\mu\frac{\left(u^r(\frac{x}{t})\right)_x}{v^r(\frac{x}{t})}\right)\in C(0,+\infty;H^2(\mathbb{R})).
\end{aligned}
\end{equation}
In addition, as $t\rightarrow+\infty$,
\begin{equation}\label{eq38}
\begin{aligned}
\sup\limits_{x\in\mathbb{ R}}&\Big|v(t,x)-\left(v^r(x/t)+\widetilde{v}^S(x-\sigma t-X(t))-v_m\right),\\
&u(t,x)-\left(u^r(x/t)+\widetilde{u}^S(x-\sigma t-X(t))-u_m\right),\\
&\Pi(t,x)-\left(\widetilde{\Pi}^S(x-\sigma t-X(t))+\mu \frac{\left(u^r(\frac{x}{t})\right)_x}{v^r(\frac{x}{t})} \right)\Big|
 \rightarrow0,
\end{aligned}
\end{equation}
and
\begin{equation}\label{eq39}
\lim\limits_{t\rightarrow+\infty}|\dot{X}(t)|=0.
\end{equation}
\end{theorem}
\begin{remark}
If we let $\delta_R=0$, the results in Theorem \ref{th1} holds true in the case of a single viscous shock. Therefore, we provide a different method to show the nonlinear stability of viscous shocks compare to that in \cite{MNS}.
\end{remark}
\begin{remark}
According to \eqref{eq39}, the shift function $X(t)$ satisfies
\begin{align*}
\LIM t {+\infty} \frac{X(t)}{t}=0.
\end{align*}
This shows that the shift function $X(t)$ will keeps the original traveling wave profile time-asymptotically, see \cite{WY}.
\end{remark}
\begin{remark}
The notation  $R_1(v_m,u_m)$ and $S_2(v_+,u_+)$ denotes the usual 1-rarefaction wave solution and 2-shock wave solution of Riemann problem for $p$-system \eqref{eq3}, respectively.
\end{remark}

Furthermore, based on the uniform estimates of error terms, we have the following convergence theorem.
\begin{theorem}\label{th1.2}
Let $(v^{\tau}, u^{\tau}, \Pi^{\tau})$ be the global solutions obtained in Theorem \ref{th1}, then there exists functions $(v^0 ,u^0)\in L^{\infty}\left((0, \infty);H^2\right)$ and $\Pi^0\in L^2\left((0, \infty);H^2\right)$, then, as $\tau\rightarrow0$, we have
\begin{align*}
(v^{\tau}, u^{\tau})\rightharpoonup (v^0, u^0) \qquad weakly-*\quad in \quad L^{\infty}\left((0, \infty);H^2\right),\\
\Pi^{\tau}\rightharpoonup \Pi^0 \qquad weakly- \quad in \quad L^2\left((0, \infty);H^2\right),
\end{align*}
where $(v^0, u^0)$ is the solution to the classical one-dimensional isentropic compressible Navier-Stokes equations \eqref{1}, with initial value $(v^0_0, u^0_0)$ is the weak limit of $(v^{\tau}_0, u^{\tau}_0)$ in \eqref{eq2}. Moreover,
\[
\Pi^0=\mu\frac{(u^0)_{x}}{v^0}.
\]
\end{theorem}

The paper is organized as follows. Some basic concept, including relative quantities, rarefaction wave, viscous shock wave and $a$-contraction with shifts theory  are given in Section 2. In Section 3, we reformulate the original problem and present a proposition of the a priori estimates (Proposition \ref{p1}) which gives the proof of Theorem \ref{th1} immediately. In Section 4, we give a proof  of Proposition \ref{p1}. Finally, in Section 5, we prove that the solutions of  relaxed system \eqref{eq1} converges globally in time to  that of  classical system \eqref{1}.

\textbf{Notations:}  $L^p(\mathbb R)$ and $W^{s,p}(\mathbb R)$  ($1\le p \le\infty$) denote the  usual Lebesgue  and Sobolev spaces over $\mathbb R$ with the norm $\|\cdot \|_{L^p}$ and $\|\cdot\|_{W^{s,p}}$, respectively. Note that, when $s=0$, $W^{0,p}=L^p$. For $p=2$, $W^{s, 2}$ are abbreviated to $H^s$ as usual.
Let $T$ and $B$ be a positive constant and a Banach space, respectively. $C^k(0,T; B)(k \ge 0 )$ denotes the space of $B$-valued $k$-times continuously differentiable functions on $[0,T]$, and $L^p(0,T; B)$ denotes the space of $B$-valued $L^p$-functions on $[0,T]$. The corresponding space $B$-valued functions on  $[0,\infty)$ are defined similarly.

\section{Preliminaries}

\subsection{Relative quantities}
We first introduce some  relative quantities. For any function $F$ defined on $R^+$, we define the associated relative quantity for $v,w\in R^+$ as
\[F(v|w)=F(v)-F(w)-F^{\prime}(w)(v-w).\]
For the pressure function $p(v)=v^{-\gamma}$, and the potential energy $H(v)=v^{1-\gamma}/(\gamma-1)$, we have the following lemma, see \cite{KV3, WY}.
\begin{lemma}\label{Le1}
For given constants $\gamma>1$, and $v_->0$, there exist constants $C,\delta_{\ast}>0$, such that the following holds true.
\begin{itemize}
\item[1)] For any $v,w$ such that $0<w<2v_-,0<v<3v_-$,
\[
|v-w|^2\leq CH(v|w),
\]
\[
|v-w|^2\leq Cp(v|w).
\]
\item[2)] For any $v,w>v_-/2$,
\[
|p(v)-p(w)|\leq C|v-w|.
\]
\item[3)] For any $0<\delta<\delta_{\ast}$, and for any $(v,w)\in \mathbb{R}^2_+$ satisfying $|p(v)-p(w)|<\delta$, and $|p(w)-p(v_-)|<\delta$, the following holds true:
\[
p(v|w)\leq\left(\frac{\gamma+1}{2\gamma}\frac{1}{p(w)}+c\delta\right)|p(v)-p(w)|^2,
\]
\[
H(v|w)\geq\frac{p(w)^{-\frac{1}{\gamma}-1}}{2\gamma}|p(v)-p(w)|^2-\frac{1+\gamma}{3\gamma^2}p(w)^{-\frac{1}{\gamma}-2}(p(v)-p(w))^3,
\]
\[
H(v|w)\leq\left(\frac{p(w)^{-\frac{1}{\gamma}-1}}{2\gamma}+c\delta\right)|p(v)-p(w)|^2.
\]

\end{itemize}
\end{lemma}

\subsection{Rarefaction wave, viscous shock wave and composite waves}

We first recall the 1-rarefaction wave by considering the Riemann problem for the inviscid Burgers equation:
\begin{equation}\label{eq19}
\begin{cases}
w_t+ww_x=0,\\
w(0,t)=w^r_0(x)=\begin{cases}
w_-,\quad x<0,\\
w_m,\quad x>0.\\
\end{cases}
\end{cases}
\end{equation}
If $w_-<w_m$, then \eqref{eq19} has a  solution $w^r(t,x)=w^r(x/t)$ given by
\[
w^r(t,x)=w^r(x/t)=
\begin{cases}
w_-,\quad x<w_-t,\\
\frac{x}{t},\quad w_-t\leq x\leq w_mt,\\
w_m,\quad x>w_mt.
\end{cases}
\]
The 1-rarefaction wave $(v^r,u^r)(t,x)=(v^r,u^r)(x/t)$ to the Riemann problem \eqref{eq3}-\eqref{eq4} is given by
\[
\left.
\begin{aligned}
&\lambda_1(v^r(\frac{x}{t}))=w^r(\frac{x}{t})),\\
&z_1(v^r(\frac{x}{t}),u^r(\frac{x}{t}))=z_1(v_-,u_-)=z_1(v_m,u_m),
\end{aligned}
\right.
\]
where $\lambda_1(v)=-\sqrt{-p^{\prime}(v)}$ and $z_1(v,u)=u+\int^v\lambda_1(s)ds$ is called the 1-Riemann invariant to $p$-system.
It is easy to check that the 1-rarefaction wave $(v^r,u^r)(x/t)$ satisfies the Euler system $a.e.$ for $t>0$,
\[
\begin{cases}
v^r_t-u^r_x=0,\\
u^r_t+p(v^r)_x=0.
\end{cases}
\]
Let $\delta_R:=|v_m-v_-|$ denote the strength of the rarefaction wave. Notice that $\delta_R\sim|u_m-u_-|$.

Now, we show the existence of traveling wave solutions (viscous shock wave) connecting $(v_m,u_m)$ and $(v_+,u_+)$ for systems \eqref{eq1}.
Let $\xi=x-\sigma t$ with $\sigma^2=\frac{p(v_m)-p(v_+)}{v_+-v_m}$ be the speed of shock wave. Assume the functions $(\widetilde{u}^S,\widetilde{v}^S,\widetilde{\Pi}^S)(\xi)$ satisfy
\begin{equation}\label{eq24}
\begin{aligned}
(\widetilde{u}^S,\widetilde{v}^S,\widetilde{\Pi}^S)(\xi)\rightarrow(v_+,u_+,0),\quad(\xi\rightarrow+\infty),\\
(\widetilde{u}^S,\widetilde{v}^S,\widetilde{\Pi}^S)(\xi)\rightarrow(v_m,u_m,0),\quad(\xi\rightarrow-\infty).
\end{aligned}
\end{equation}
 Plugging the form $(\widetilde{u}^S,\widetilde{v}^S,\widetilde{\Pi}^S)(\xi)$ into system \eqref{eq1}, we have the following ordinary differential equations
\begin{equation}\label{eq25}
\begin{cases}
-\sigma\widetilde{v}^S_{\xi}-\widetilde{u}^S_{\xi}=0,\\
-\sigma\widetilde{u}^S_{\xi}+p(\widetilde{v}^S)_{\xi}=\widetilde{\Pi}^S_{\xi},\\
-\sigma\tau\widetilde{\Pi}^S_{\xi}+\widetilde{v}^S\widetilde{\Pi}^S=\mu\widetilde{u}^S_{\xi},
\end{cases}
\end{equation}
with the far field condition \eqref{eq24}. Then integrating the equations $\eqref{eq25}_1$ and $\eqref{eq25}_2$ with respect to $\xi$, it holds
\begin{equation}\label{eq26}
\begin{cases}
\sigma\widetilde{v}^S+\widetilde{u}^S=\sigma v_m+u_m=\sigma v_++u_+,\\
\widetilde{\Pi}^S=-\sigma(\widetilde{u}^S-u_m)+(p(\widetilde{v}^S)-p(v_m)).
\end{cases}
\end{equation}
Substituting \eqref{eq26} and $\eqref{eq25}_1$ into $\eqref{eq25}_3$, we derive that
\begin{equation}\label{eq27}
\widetilde{v}^S_{\xi}=\frac{\widetilde{v}^Sh(\widetilde{v}^S)}{\mu\sigma+\tau\sigma h^{\prime}(\widetilde{v}^S)},
\end{equation}
where $h(\widetilde{v}^S)=\sigma^2(v_m-\widetilde{v}^S)+(p(v_m)-p(\widetilde{v}^S))$.

Let $(v_m,u_m)\neq(v_+,u_+)$ and under the assumption
\begin{equation}\label{eq28}
\tau<\inf\limits_{\widetilde{v}^S\in[v_m,v_+]}\frac{\mu}{|h^{\prime}(\widetilde{v}^S)|},
\end{equation}
we get the following lemma, see \cite{MNS, ZWH, KV3, WY}.
\begin{lemma}\label{Le2}
For any states $(v_m,u_m)$, $(v_+,u_+)$, and $\sigma>0$ satisfying R-H condition and Lax condition, there exists a positive constant $C$ independent of $\tau$ such that the following is true. The traveling wave solution $(\widetilde{u}^S,\widetilde{v}^S,\widetilde{\Pi}^S)(\xi)$ connecting $(v_m,u_m,0)$ and $(v_+,u_+,0)$ exists uniquely and satisfies
\[\widetilde{v}^S_{\xi}>0,\qquad\widetilde{u}^S_{\xi}<0,\]
and
\[
\begin{aligned}
&|\widetilde{v}^S(\xi)-v_m|\leq C\delta_S e^{-C\delta_S|\xi|},
|\widetilde{u}^S(\xi)-u_m|\leq C\delta_S e^{-C\delta_S|\xi|},\quad \xi<0,\\
&|\widetilde{v}^S(\xi)-v_+|\leq C\delta_S e^{-C\delta_S|\xi|},
|\widetilde{u}^S(\xi)-u_+|\leq C\delta_S e^{-C\delta_S|\xi|},\quad \xi>0,\\
&|\widetilde{\Pi}^S|\leq C\delta_S^2 e^{-C\delta_S|\xi|},\quad \quad
|(\widetilde{v}^S_{\xi},\widetilde{u}^S_{\xi},\widetilde{\Pi}^S_{\xi})|\leq C\delta_S^2e^{-C\delta_S|\xi|},\quad \forall\xi\in\mathbb{R},\\
&|\partial^k_{\xi}(\widetilde{v}^S,\widetilde{u}^S,\widetilde{\Pi}^S)|\leq C|(\widetilde{v}^S_{\xi},\widetilde{u}^S_{\xi},\widetilde{\Pi}^S_{\xi})| \quad \forall\xi\in\mathbb{R},
\end{aligned}
\]
for $k=2,3,4$, where $\delta_S$ denote the strength of the shock as $\delta_S:=|p(v_+)-p(v_m)|\sim|v_+-v_m|\sim|u_+-u_m|$.
\end{lemma}
\begin{remark}
Substituting $\eqref{eq25}_1$ and $\eqref{eq25}_2$ into $\eqref{eq25}_3$, we derive that
\[
\widetilde{\Pi}^S=\frac{-\sigma^3\tau+\sigma\tau p^{\prime}(\widetilde{v}^S)-\mu\sigma}{\widetilde{v}^S}\widetilde{v}^S_{\xi},
\]
and so
\begin{equation}\label{eq29}
|\widetilde{\Pi}^S|\sim|\widetilde{v}^S_{\xi}|.
\end{equation}
\end{remark}
\begin{remark}
Let $\tau=0$ in \eqref{eq27}, we have
\begin{equation}\label{eq30}
\widetilde{v}^S_{\xi}=\frac{\widetilde{v}^S[\sigma^2(v_m-\widetilde{v}^S)+(p(v_m)-p(\widetilde{v}^S))]}{\mu\sigma},
\end{equation}
which is the same as that obtained by Matsumura and Nishihara in \cite{MNS}.
\end{remark}

Finally, we consider the composite waves of the superposition  of a viscous shock and a rarefaction. We first remark that, for any given end states $(v_{\pm},u_{\pm},0)\in\mathbb{R}^+\times\mathbb{R}\times\mathbb{R}$ in \eqref{eq2}, there exists a unique intermediate state $(v_m,u_m,0)$ such that $(v_-,u_-,0)$ is connected to $(v_m,u_m,0)$ by the rarefaction wave, and the traveling wave connects $(v_m,u_m,0)$ and $(v_+,u_+,0)$, see \cite{SMO}. The composite waves of the superposition of a viscous shock and a rarefaction is defined as follows:
\begin{equation}\label{eq31}
\left(v^r(\frac{x}{t})+\widetilde{v}^S(x-\sigma t)-v_m,u^r(\frac{x}{t})+\widetilde{u}^S(x-\sigma t)-u_m,\widetilde{\Pi}^S(x-\sigma t)+\mu\frac{\left(u^r(\frac{x}{t})\right)_x}{v^r(\frac{x}{t})}\right).
\end{equation}

\subsection{Construction of approximate rarefaction wave}
We construct smooth approximate solution of the $1$-rarefaction wave by using the smooth solutions to the Burgers equation:
\begin{equation}\label{eq32}
\begin{cases}
w_t+ww_x=0,\\
w(0,x)=w_0(x)=\frac{w_m+w_-}{2}+\frac{w_m-w_-}{2}k_q\int_0^{\epsilon x}(1+y^2)^{-q}\dif y,
\end{cases}
\end{equation}
where $\epsilon>0$ is a constant as well as $k_q$  satisfying $k_q\int_0^{\infty}(1+y^2)^{-q}\dif y=1$ for $q>\frac{3}{2}$. Then we have the following lemma.
\begin{lemma}\label{Le3}
Let $w_-<w_m$ and set $\widetilde{w}=w_m-w_-$. then there exists a smooth solution $w(t,x)$ of the system \eqref{eq32} satisfying
\begin{itemize}
\item[1)] $w_-<w(t,x)<w_m,w_x>0$ for $x\in \mathbb{R}$ and $t\geq0$.
\item[2)] The following estimates hold for all $t>0$ and $1\leq p\leq\infty$:
\begin{equation}\nonumber
\begin{aligned}
&\|w_x(t,\cdot)\|_{L^p}\leq C\min\left(\epsilon^{1-\frac{1}{p}}\widetilde{w},\widetilde{w}^{\frac{1}{p}}t^{-1+\frac{1}{p}}\right),\\
&\|\partial^k_xw(t,\cdot)\|_{L^p}\leq C\min\left(\epsilon^{2-\frac{1}{p}}\widetilde{w},\epsilon^{(1-\frac{1}{2p})(1-\frac{1}{p})}\widetilde{w}^{-\frac{p-1}{2pq}}t^{-1-\frac{p-1}{2pq}}\right),
\end{aligned}
\end{equation}
where $k=2, 3, 4$.
\item[3)] If $w_->0$, then it holds that $\forall t\geq0,x\in \mathbb{R}$,
 \begin{equation}\nonumber
\begin{aligned}
&|w(t,x)-w_-|\leq C\widetilde{w}(1+(\epsilon x)^2)^{-\frac{q}{3}}(1+(\epsilon w_-t)^2)^{-\frac{q}{3}},\\
&|w_x(t,x)|\leq C\epsilon\widetilde{w}(1+(\epsilon x)^2)^{-\frac{q}{2}}(1+(\epsilon w_-t)^2)^{-\frac{q}{2}}.
\end{aligned}
\end{equation}
\item[4)] If $w_m<0$, then it holds that $\forall t\geq0,x\in \mathbb{R}$,
 \begin{equation}\nonumber
\begin{aligned}
&|w(t,x)-w_m|\leq C\widetilde{w}(1+(\epsilon x)^2)^{-\frac{q}{3}}(1+(\epsilon w_mt)^2)^{-\frac{q}{3}},\\
&|w_x(t,x)|\leq C\epsilon\widetilde{w}(1+(\epsilon x)^2)^{-\frac{q}{2}}(1+(\epsilon w_mt)^2)^{-\frac{q}{2}}.
\end{aligned}
\end{equation}
\item[5)] $\lim\limits_{t\rightarrow\infty}\sup\limits_{x\in R}|w(t,x)-w^r(x/t)|=0$.
\end{itemize}
\end{lemma}

Next, we construct the smooth approximate 1-rarefaction wave $(\widetilde{v}^R,\widetilde{u}^R)(t,x)$ of  $(v^r,u^r)(x/t)$ by
\begin{equation}\label{eq33}
\begin{aligned}
&\lambda_1(v_-)=w_-,\lambda_1(v_m)=v_m,\\
&\lambda_1(\widetilde{v}^R(t,x))=w(t,x),\\
&z_1(\widetilde{v}^R,\widetilde{u}^R)(t,x)=z_1(v_-,u_-)=z_1(v_m,u_m),
\end{aligned}
\end{equation}
where $w(t,x)$ is the smooth solution to the Burgers equation in \eqref{eq32}. It is easy to check that the above approximate rarefaction wave $(\widetilde{v}^R,\widetilde{u}^R)$ satisfy the Euler system for $t>0$:
\begin{equation}\label{eq34}
\begin{cases}
\widetilde{v}^R_t-\widetilde{u}^R_x=0,\\
\widetilde{u}^R_t+p(\widetilde{v}^R)_x=0.
\end{cases}
\end{equation}

For estimates of $(\widetilde{v}^R,\widetilde{u}^R)$ , we have the following lemma, see \cite{MNR1, MNR2, NNK, XFH}

\begin{lemma}\label{Le4}
The smooth approximate 1-rarefaction wave $(\widetilde{v}^R,\widetilde{u}^R)(t,x)$ defined in \eqref{eq33} satisfies the following properties.
\begin{itemize}
\item[1)] $\widetilde{v}^R_t>0$, $\forall(t,x)\in \mathbb{R}^+\times \mathbb{R}$.
\item[2)] There exists a constant C such that
$$|\widetilde{v}^R_x|\leq C\widetilde{v}^R_t, \quad \widetilde{v}^R_t\leq C\epsilon\delta_R, \quad \forall(t,x)\in \mathbb{R}^+\times \mathbb{R}.$$
\item[3)] $\forall t>0$,
$$\|\widetilde{v}^R_x\|_{L^p},\|\widetilde{u}^R_x\|_{L^p}\leq C\min\{\delta_R\epsilon^{1-\frac{1}{p}},\delta_R^{\frac{1}{p}}(1+t)^{-1+\frac{1}{p}}\}.$$
\item[4)] $\forall t>0$ and $k=2,3,4$,
\begin{align*}
&\|\partial_x^k\widetilde{v}^R\|_{L^p},\|\partial^k_x\widetilde{u}^R\|_{L^p}\leq C\left(\delta_R^{-\frac{p-1}{2pq}}\epsilon^{(1-\frac{1}{2q})(1-\frac{1}{p})}(1+t)^{-1-\frac{p-1}{2pq}}
+\delta_R^{\frac{1}{p}}(1+t)^{-2+\frac{1}{p}}\right),
\end{align*}
and $\forall p\geq1$,
$$\int_0^{\infty}(\|\partial_x^k\widetilde{v}^R\|_{L^p},\partial^k_x\|\widetilde{u}^R\|_{L^p})dt\leq C\delta_R^{-\frac{p-1}{2pq}}.$$
\item[5)] For $x\geq0,t\geq0,$ it holds that
\begin{align*}
&|(\widetilde{v}^R,\widetilde{u}^R)(t,x)-(v_m,u_m)|\leq C\delta_Re^{-2(|x|+|\lambda_1(v_m)|t)},\\
&|(\widetilde{v}^R_x,\widetilde{u}^R_x)(t,x)|\leq C\delta_Re^{-2(|x|+|\lambda_1(v_m)|t)}.
\end{align*}
\item[6)] For $x\leq\lambda_1(v_-)t<0$ and $t\geq0$, it holds that
\begin{align*}
&|(\widetilde{v}^R,\widetilde{u}^R)(t,x)-(v_-,u_-)|\leq C\delta_Re^{-2|x-\lambda_1(v_-)t)|},\\
&|(\widetilde{v}^R_x,\widetilde{u}^R_x)(t,x)|\leq C\delta_Re^{-2|x-\lambda_1(v_-)t)|}.
\end{align*}
\item[7)] $\lim\limits_{t\rightarrow\infty}\sup\limits_{x\in R}|(\widetilde{v}^R,\widetilde{u}^R)(t,x)-(v^r,u^r)(x/t)|=0$.
\end{itemize}
\end{lemma}

\subsection{Construction of shift}
As in \cite{WY}, the $L^2$-estimates in Sec.4 depend closely  on the shift function $X(t)$, which is defined as a solution to the ODE:
\begin{equation}\label{eq35}
\begin{cases}
\begin{aligned}
\dot{X}(t)=-\frac{M}{\delta_S}[&\int_\mathbb{R}\frac{a(\xi-X)}{\sigma}\widetilde{u}^S_{\xi}(\xi-X)(p(v)-p(\widetilde{v}_{-X}))\dif \xi\\
           -&\int_\mathbb{R}a(\xi-X)p(\widetilde{v}^S(\xi-X)_{\xi}(v-\widetilde{v}_{-X})\dif \xi],\\
\end{aligned}\\
X(0)=0,
\end{cases}
\end{equation}
where $M$ is a specific constant chosen as $M:=\frac{5(\gamma+1)\sigma_m^3}{8\gamma p(v_m)}$ with $\sigma_m:=\sqrt{-p^{\prime}(v_m)}$ and $\widetilde{v}_{-X}(t,\xi)=\widetilde{v}^R(t,\xi+\sigma t)+\widetilde{v}^S(\xi-X(t))-v_m$.  The weight function $a$  is defined by
\begin{equation}\label{eq43}
 a(\xi):=1+\frac{\lambda}{\delta_S}(p(v_m)-p(\widetilde{v}^S(\xi)),
\end{equation}
where the constant $\lambda$ is chosen to be so small but far bigger than $\delta_S$ such that
\[\delta_S\ll\lambda\leq C\sqrt{\delta_S}.\]
Notice that
\begin{equation}\label{55}
1<a(\xi)<1+\lambda,
\end{equation}
and
\begin{equation}\label{56}
a^{\prime}(\xi)=-\frac{\lambda}{\delta_S}p^{\prime}(\widetilde{v}^S)\widetilde{v}^S_{\xi}>0.
\end{equation}

We remark that the shift function above  is slightly different with that in \cite{WY} due to the unavailability of B-D entropy. Nonetheless,  using similar proof as in \cite{WY}, we can get that \eqref{eq35} has a unique absolutely continuous solution defined on any interval $[0,T]$ provided  $v(t,x)>0$ is bounded both above and below for any $(t,x)\in[0,T]\times \mathbb{R}$.

\section{Local solutions and a priori estimates}

To simplify our analysis, we rewrite the system \eqref{eq1} into the following system, by change of  variable  $(t,x)\rightarrow(t,\xi=x-\sigma t)$:
\begin{equation}\label{eq40}
\begin{cases}
v_t-\sigma v_{\xi}-u_{\xi}=0,\\
u_t-\sigma u_{\xi}+p_{\xi}=\Pi_{\xi},\\
\tau \Pi_t-\sigma\tau\Pi_{\xi}+v\Pi=\mu u_{\xi}.
\end{cases}
\end{equation}
By using the theory of symmetric hyperbolic system, we have the following local existence theory, see \cite{JR, TA, WY}.
\begin{theorem}\label{th3.1}
Let $\underline{v}$ and $\underline{u}$ be smooth monotone functions such that
\[
\underline{v}(x)=v_{\pm} \quad \underline{u}(x)=u_{\pm} \quad for \quad \pm x\geq1.
\]
For any constants $M_0,M_1,\kappa_1,\kappa_2,\kappa_3,\kappa_4$ with $M_1>M_0>0$ and $\kappa_1>\kappa_2>\kappa_3>\kappa_4>0$, there exists a constant $T_0>0$ such that if
\[
\begin{aligned}
&\|v_0-\underline{v}\|_{H^2}+\|u_0-\underline{u}\|_{H^2}+\sqrt{\tau}\|\Pi_0\|_{H^2}\leq M_0,\\
&0<\kappa_3\leq v_0(x)\leq\kappa_2,\qquad \forall x\in\mathbb{R},
\end{aligned}
\]
then the system \eqref{eq40} with initial condition \eqref{eq2} has a unique solution $(v, u, \Pi)$ on $[0,T_0]$ such that
\[
\begin{aligned}
(v-\underline{v}, u-\underline{u}, \Pi)\in C([0,T_0];H^2),
\end{aligned}
\]
and
\[
\|v-\underline{v}\|_{L^{\infty}(0,T_0;H^2)}+\|u-\underline{u}\|_{L^{\infty}(0,T_0;H^2)}
+\sqrt{\tau}\|\Pi\|_{L^{\infty}(0,T_0;H^2)}\leq M_1.
\]
Moreover:
\[
\kappa_4\leq v(t,x)\leq\kappa_1,\qquad \forall(t,x)\in[0,T_0]\times\mathbb{R}.
\]
\end{theorem}

Next, we will focus on the time-asymptotic stability of the solutions to \eqref{eq40} with initial data \eqref{eq2} to the superposition of the approximate rarefaction wave and the traveling wave shifted by $X(t)$ (given by \eqref{eq35}) which is give by
\begin{align}
&(\widetilde{v}_{-X},\widetilde{u}_{-X},\widetilde{\Pi}_{-X})(t,\xi) := \label{eq41}\\
&\left(\widetilde{v}^R(t,\xi+\sigma t)+\widetilde{v}^S(\xi-X(t))-v_m,
           \widetilde{u}^R(t,\xi+\sigma t)+\widetilde{u}^S(\xi-X(t))-u_m,
           \widetilde{\Pi}^S(\xi-X(t))+\mu\frac{\widetilde{u}^R_{\xi}}{\widetilde{v}^R}\right). \nonumber
\end{align}
Denote $\left(\widetilde v_0( \cdot), \widetilde u_0( \cdot), \widetilde \Pi_0( \cdot)\right)=\left(\widetilde v_{-X}, \widetilde u_{-X}, \widetilde \Pi_{-X}\right)(t, \cdot)\Big|_{t=0}$. Then,  one can easily get
\begin{align}\label{hu4.11}
\sum\limits_{\pm}\left(\|(\widetilde v_0-v_{\pm},\widetilde u_0-u_{\pm}, \sqrt{\tau}\widetilde \Pi_0)\|_{H^2(\mathbb{R}_{\pm})}\right)
\le C (\delta_S+\delta_R).
\end{align}

We first give the equations satisfied by the above composite waves. It is  easy to check that the approximate rarefaction wave $(\widetilde{v}^R(t,\xi+\sigma t),\widetilde{u}^R(t,\xi+\sigma t))$ satisfies the following system:
\begin{equation}\label{eq42}
\begin{cases}
v_t-\sigma v_{\xi}-u_{\xi}=0,\\
u_t-\sigma u_{\xi}+p(v)_{\xi}=0.
\end{cases}
\end{equation}
Therefore, using \eqref{eq25} and \eqref{eq42}, we find that the approximated combination of waves $(\widetilde{v}_{-X},\widetilde{u}_{-X},\widetilde{\Pi}_{-X})$ defined in \eqref{eq41} solves the following system:
\begin{equation}\label{43}
\begin{cases}
(\widetilde{v}_{-X})_t-\sigma (\widetilde{v}_{-X})_{\xi}+\dot{X}(t)(\widetilde{v}^S)_{\xi}^{-X}-(\widetilde{u}_{-X})_{\xi}=0,\\
(\widetilde{u}_{-X})_t-\sigma (\widetilde{u}_{-X})_{\xi}+\dot{X}(t)(\widetilde{u}^S)_{\xi}^{-X}+p(\widetilde{v}_{-X})_{\xi}=(\widetilde{\Pi}_{-X})_{\xi}+\widetilde{F_1},\\
\tau (\widetilde{\Pi}_{-X})_t-\sigma\tau (\widetilde{\Pi}_{-X})_{\xi}+\tau\dot{X}(t)(\widetilde{\Pi}^S)_{\xi}^{-X}+(\widetilde{v}_{-X})(\widetilde{\Pi}_{-X})=\mu (\widetilde{u}_{-X})_{\xi}+\widetilde{F_2},
\end{cases}
\end{equation}
where $(\widetilde{v}^S)_{\xi}^{-X}=\widetilde{v}^S_{\xi}(\xi-X(t)),(\widetilde{u}^S)_{\xi}^{-X}=\widetilde{u}^S_{\xi}(\xi-X(t)),
(\widetilde{\Pi}^S)_{\xi}^{-X}=\widetilde{\Pi}^S_{\xi}(\xi-X(t))$ and
\begin{equation}\label{44}
\begin{aligned}
&\widetilde{F_1}=p(\widetilde{v}_{-X})_{\xi}-p(\widetilde{v}^R)_{\xi}-p((\widetilde{v}^S)^{-X})_{\xi}-\left(\mu\frac{\widetilde{u}^R_{\xi}}{\widetilde{v}^R}\right)_{\xi},\\
&\widetilde{F_2}=\tau\left(\mu\frac{\widetilde{u}^R_{\xi}}{\widetilde{v}^R}\right)_t-\sigma\tau\left(\mu\frac{\widetilde{u}^R_{\xi}}{\widetilde{v}^R}\right)_{\xi}
+(\widetilde{v}^R-v_m)(\widetilde{\Pi}^S)^{-X}
+\left((\widetilde{v}^S)^{-X}-v_m\right)\left(\mu\frac{\widetilde{u}^R_{\xi}}{\widetilde{v}^R}\right).
\end{aligned}
\end{equation}
The following proposition gives the a priori estimates of the error term: $(v-\widetilde{v}_{-X}, u-\widetilde{u}_{-X}, \Pi-\widetilde{\Pi}_{-X})$.

\begin{prop}\label{p1}
Let $(v, u, \Pi)$ be local solutions given by Theorem \ref{th3.1} on $[0,T]$ for some $T>0$ and $(\widetilde{v}_{-X},\widetilde{u}_{-X},\widetilde{\Pi}_{-X})$ be defined in \eqref{eq41}.  Assume that there exist positive constants $\delta_0,\varepsilon_1$ such that both the rarefaction and shock waves strength satisfy $\delta_R,\delta_S<\delta_0$ and
\begin{equation}\label{45}
\sup_{0\le t\le T} \|(v-\widetilde{v}_{-X}, u-\widetilde{u}_{-X}, \sqrt{\tau}(\Pi-\widetilde{\Pi}_{-X}))\|_{H^2}\leq\varepsilon_1,
\end{equation}
then the following estimates hold
\begin{equation}\label{46}
\begin{aligned}
&\sup\limits_{t\in[0,T]}\left[\|v-\widetilde{v}_{-X}\|^2_{H^2}
+\|u-\widetilde{u}_{-X}\|^2_{H^2}+\tau\|\Pi-\widetilde{\Pi}_{-X}\|^2_{H^2}\right]
+\delta_S\int_0^T|\dot{X}(t)|^2\dif t\\
&+\int_0^T\left(\mathcal{G}^S(U)+\mathcal{G}^R(U)
+\|\left((v-\widetilde{v}_{-X})_{\xi}, (u-\widetilde{u}_{-X})_{\xi}\right)\|^2_{H^1}+\|\Pi-\widetilde{\Pi}_{-X}\|^2_{H^2}\right)\dif t\\
&\leq C_0\left(\|v_0-\widetilde{v}_0(\cdot)\|^2_{H^2}+\|u_0-\widetilde{u}_0(\cdot)\|^2_{H^2}
+\tau\|\Pi_0-\widetilde{\Pi}_0(\cdot)\|^2_{H^2}\right)+C_0\delta_R^{\theta},
\end{aligned}
\end{equation}
where $C_0$ independent of $T$ and $\tau$, $\theta=\min\{\frac{1}{2}, \frac{3}{2}-\frac{1}{q}\}$ and
\begin{equation}\label{47}
\mathcal{G}^S(U):=\int_\mathbb{R}(\widetilde{v}^S)_{\xi}^{-X} |v-\widetilde{v}_{-X}|^2\dif \xi,
\mathcal{G}^R(U):=\int_\mathbb{R} \widetilde{u}^R_{\xi} |v-\widetilde{v}_{-X}|^2\dif \xi.
\end{equation}
In addition, by \eqref{eq35},
\begin{equation}\label{48}
|\dot{X}(t)|\leq C_0\|(v-\widetilde{v}_{-X})(t,\cdot)\|_{L^{\infty}},\qquad\forall t\leq T.
\end{equation}
\end{prop}
We will give the proof of Proposition \ref{p1} in Section 4.

Now, by using Theorem \ref{th3.1} and Proposition \ref{p1}, we are able to prove Theorem \ref{th1}.

{\bf{Proof of Theorem \ref{th1}}}:   From Theorem \ref{th3.1}, the system \eqref{eq40} with initial condition \eqref{eq2} has a unique solution $(v, u, \Pi)\in C^0([0, T], H^2)$ for some $T>0$.  By definition of $\underline v, \underline u$, we have
\[
\sum\limits_{\pm}\left(\|(\underline{v}-v_{\pm},\underline{u}-u_{\pm})\|_{L^2(\mathbb{R}_{\pm})}\right)
+\|((\underline{v})_x,(\underline{u})_x)\|_{H^1}<C_\ast(\delta_S+\delta_R).
\]
Then, using \eqref{hu4.11} and  the assumption \eqref{eq36} in Theorem \ref{th1}, we have
\begin{align}\label{hu4.11-2}
\|(\underline{v}-\widetilde{v}_0(\xi),\underline{u}-\widetilde{u}_0(\xi))\|_{H^2}
+\sqrt{\tau}\|\widetilde{\Pi}_0(\xi)\|_{H^2}
<C_{\ast}(\delta_S+\delta_R)
\end{align}
and
\begin{align}\label{hu4.11-3}
\|(v_0-\underline{v},u_0-\underline{u})\|_{H^2}
+\sqrt{\tau}\|\Pi_0\|_{H^2}
<\varepsilon_0+C_{\ast}(\delta_S+\delta_R).
\end{align}

Next, let energy term $E(t)$ defined as follows
\[
E(t)=\sup_{0\leq s \leq t}\left\|\left(v-\widetilde{v},u-\widetilde{u},\sqrt{\tau}(\Pi-\widetilde{\Pi})\right)\right\|_{H^2}.
\]
Then, from \eqref{hu4.11-2} and \eqref{hu4.11-3},  for sufficiently small $\varepsilon_0, \delta_S, \delta_R$, we get
\begin{align*}
E(0)\le\varepsilon_0+2C_{\ast}(\delta_S+\delta_R)<\frac{\varepsilon_1}{2}.
\end{align*}
Next we shall show $E(T)\le \frac{\varepsilon_1}{2}$  for sufficiently small $\varepsilon_0, \delta_S, \delta_R$, which  justify the a priori assumption \eqref{45}.
From Proposition \ref{p1}, if
\[
E(T)<\varepsilon_1,
\]
then, we have
\[
E(t)\leq C_0E(0)+C_0\delta_R^{\theta}\le C_0 (\varepsilon_0+2C_{\ast}(\delta_S+\delta_R))+C_0\delta_R^{\theta}.
\]
Now, choosing $\delta_S, \delta_R$ sufficiently small, we can get
\[
E(t)<\frac{\varepsilon_1}{2},
\]
Thus, the a priori assumption \eqref{45} is closed.
Therefore, based on the local exist theorem (Theorem \ref{th3.1}) and a priori estimates (Proposition \ref{p1}) we get a global solution in time by the classical continuation methods.

In addition, using system \eqref{50} below and  \eqref{46}, we can get
$$\int_0^{\infty}\left\|\partial_{\xi}(v-\widetilde{v}, u-\widetilde{u}, \Pi-\widetilde{\Pi})\right\|_{L^2}^2\dif t \le C$$
and
$$\int_0^{\infty} \Big|\frac{\dif}{\dif t}\left\|\partial_{\xi}(v-\widetilde{v}, u-\widetilde{u}, \Pi-\widetilde{\Pi})\right\|_{L^2}^2 \Big |\dif t \le C.$$
Thus, combining the interpolation inequality, we get
\[
\lim\limits_{t\rightarrow\infty}\left\|(v-\widetilde{v}, u-\widetilde{u}, \Pi-\widetilde{\Pi})\right\|_{L^{\infty}}=0.
\]
Furthermore, using the definition of shift $X(t)$ (see \eqref{eq35}), we have
\[
|\dot{X}(t)|\leq C\left\|v-\widetilde{v}\right\|_{L^{\infty}}\rightarrow0\quad as \quad t\rightarrow\infty.
\]
Therefore,  the proof of Theorem \ref{th1} is finished.
\section{Proof of Proposition \ref{p1}}
In this section, we establish the a priori estimates  of local solutions and thus give a proof of Proposition \ref{p1}.
For simplicity, we use the notation $(\widetilde{v},\widetilde{u},\widetilde{\Pi})(t,\xi)$ instead of $(\widetilde{v}_{-X},\widetilde{u}_{-X},\widetilde{\Pi}_{-X})$.
Using \eqref{eq40} and \eqref{43}, we derive the equations for the error term $(v-\widetilde v, u- \widetilde u, \Pi-\widetilde \Pi)$ as follows:
\begin{equation}\label{50}
\begin{cases}
(v-\widetilde{v})_t-\sigma(v-\widetilde{v})_{\xi}-\dot{X}(t)(\widetilde{v}^S)_{\xi}^{-X}-(u-\widetilde{u})_{\xi}=0,\\
(u-\widetilde{u})_t-\sigma(u-\widetilde{u})_{\xi}-\dot{X}(t)(\widetilde{u}^S)_{\xi}^{-X}+(p(v)-p(\widetilde{v}))_{\xi}=(\Pi-\widetilde{\Pi})_{\xi}-F_1,\\
\tau(\Pi-\widetilde{\Pi})_t-\sigma\tau(\Pi-\widetilde{\Pi})_{\xi}-\tau\dot{X}(t)(\widetilde{\Pi}^S)_{\xi}^{-X}+(v\Pi-\widetilde{v}\widetilde{\Pi})=\mu(u-\widetilde{u})_{\xi}-F_2,
\end{cases}
\end{equation}
where
\[
F_1=p(\widetilde{v})_{\xi}-p(\widetilde{v}^R)_{\xi}-p((\widetilde{v}^S)^{-X})_{\xi}
-\left(\mu\frac{\widetilde{u}^R_{\xi}}{\widetilde{v}^R}\right)_{\xi},
\]
and $F_2=\widetilde{F_2}$.\\
To prove Proposition \ref{p1}, we need to do  the lower-order, higher-order and dissipation estimates, respectively.
\subsection{$L^2$ estimates}
The $L^2$ estimates are given in the following lemma.
\begin{lemma}\label{Le5}
Under the hypotheses of Proposition \ref{p1}, there exists $C>0$ (independent of $\tau$ and $T$) such that for all $t\in(0,T]$,
\[
\begin{aligned}
&\|v-\widetilde{v}\|^2_{L^2}+\|u-\widetilde{u}\|^2_{L^2}
 +\tau\|\Pi-\widetilde{\Pi}\|^2_{L^2}+\delta_S\int_0^t|\dot{X}(t)|^2\dif t+\int_0^t\left(\mathcal{G}^S(U)+\mathcal{G}^R(U)+G(U)\right)\dif t\\
 &\leq C\left(\|v_0-\widetilde{v}_0(\cdot)\|^2_{L^2}+\|u_0-\widetilde{u}_0(\cdot)\|^2_{L^2}
 +\tau\|\Pi_0-\widetilde{\Pi}_0(\cdot)\|^2_{L^2}\right)+C\int_0^t\int_\mathbb{R}\left|\partial_{\xi}(v-\widetilde{v})\right|^2\dif\xi \dif t+C\delta_R^{\theta},
\end{aligned}
\]
where $\theta=\min\{\frac{1}{2}, \frac{3}{2}-\frac{1}{q}\},$
\[
 G(U):=\int_\mathbb{R}\frac{v}{\mu}|\Pi-\widetilde{\Pi}|^2\dif \xi,
\]
 $\mathcal G^{S}(U)$ and $\mathcal G^{R}(U)$ are given in \eqref{47}.
\end{lemma}
Before we prove Lemma \ref{Le5}, we need some auxiliary lemmas. Firstly, using lemmas \ref{Le2} and \ref{Le4}, we have the following estimates. \begin{lemma}\label{Le9}
Let $X$ be the shift defined by \eqref{eq35}. Under the same hypotheses as in Proposition \ref{p1}, the following estimates hold: $\forall t\leq T$,
\begin{align*}
&\|(\widetilde{v}^S)^{-X}_{\xi}(\widetilde{v}^R-v_m)\|_{L^2}
+\|(\widetilde{v}^S)^{-X}_{\xi}\widetilde{v}^R_{\xi}\|_{L^2}\leq C\delta_R\delta_S^{3/2}e^{-C\delta_St},\\
&\|\widetilde{v}^R_{\xi}((\widetilde{v}^S)^{-X}-v_m)\|_{L^2}\leq C\delta_R\delta_Se^{-C\delta_St}.
\end{align*}
\end{lemma}
The proof of the above lemma is very similar to that in \cite{WY}, so we omit the proof here. (see Section 4, Lemma 4.2 in \cite{WY} for details) .

Let $U:=(u, v, \Pi)^T$  and $\widetilde{U}:=(\widetilde{u}, \widetilde{v}, \widetilde{\Pi})^T$. To show Lemma \ref{Le5} hold, we first define a relative entropy quantity:
\begin{equation}\label{58}
\eta(U|\widetilde{U})=\frac{|u-\widetilde{u}|^2}{2}+H(v|\widetilde{v})+\frac{\tau|\Pi-\widetilde{\Pi}|^2}{2\mu},
\end{equation}
where $H(v)=\frac{v^{-\gamma+1}}{\gamma-1}$, $i.e.$, $H^{\prime}(v)=-p(v)$. Then, the following lemma gives the estimate of the above relative entropy weighted by $a(\xi)$ defined in \eqref{eq43} with a shift $X$.
\begin{lemma}\label{Le6}
We have
\begin{equation}\label{60}
\frac{\dif}{\dif t}\int_\mathbb{R}a^{-X}(\xi)\eta(U(t,\xi)|\widetilde{U}(t,\xi))\dif \xi=\dot{X}(t)Y(U)+J^{bad}(U)-J^{good}(U),
\end{equation}
where
\begin{align*}
Y(U):=&-\int_\mathbb{R}a_{\xi}^{-X}\eta(U|\widetilde{U})\dif \xi+
\int_\mathbb{R}a^{-X}(\widetilde{u}^S)^{-X}_{\xi}(u-\widetilde{u})\dif \xi\\
    &-\int_\mathbb{R}a^{-X}p^{\prime}(\widetilde{v})(\widetilde{v}^S)^{-X}_{\xi}(v-\widetilde{v})\dif \xi
    +\int_\mathbb{R}a^{-X}\frac{\tau}{\mu}(\widetilde{\Pi}^S)^{-X}_{\xi}(\Pi-\widetilde{\Pi})\dif \xi,
\end{align*}
\begin{align*}
   J^{bad}(U):=&\frac{1}{2\sigma}\int_\mathbb{R}a_{\xi}^{-X}((p(v)-p(\widetilde{v}))^2\dif\xi
   +\sigma\int_\mathbb{R}a^{-X}(\widetilde{v}^S)^{-X}_{\xi}p(v|\widetilde{v})\dif\xi\\
      &-\frac{1}{\sigma}\int_\mathbb{R}a^{-X}_{\xi}(\Pi-\widetilde{\Pi})(p(v)-p(\widetilde{v}))\dif\xi
      -\int_\mathbb{R}a^{-X}\frac{\widetilde{\Pi}}{\mu}(\Pi-\widetilde{\Pi})(v-\widetilde{v})\dif\xi\\
    &+\frac{1}{2\sigma}\int_\mathbb{R}a^{-X}_{\xi}(\Pi-\widetilde{\Pi})^2\dif\xi
    -\int_\mathbb{R}a^{-X}(u-\widetilde{u})F_1\dif\xi-\int_\mathbb{R}a^{-X}\frac{(\Pi-\widetilde{\Pi})}{\mu}F_2\dif\xi,
\end{align*}
\begin{align*}
J^{good}(U):=&\frac{\sigma}{2}\int_\mathbb{R}a_{\xi}^{-X}
\left(u-\widetilde{u}-\frac{p(v)-p(\widetilde{v})}{\sigma}+\frac{\Pi-\widetilde{\Pi}}{\sigma}\right)^2\dif\xi
+\sigma\int_\mathbb{R}a^{-X}_{\xi}H(v|\widetilde{v})\dif\xi\\
     &+\int_\mathbb{R}a^{-X}\widetilde{u}^R_{\xi}p(v|\widetilde{v})\dif\xi
     +\int_\mathbb{R}a^{-X}\frac{v}{\mu}(\Pi-\widetilde{\Pi})^2\dif\xi
     +\frac{\sigma\tau}{2\mu}\int_\mathbb{R}a^{-X}_{\xi}(\Pi-\widetilde{\Pi})^2\dif\xi.
\end{align*}
\end{lemma}
\begin{proof}
By the definition of the relative entropy with \eqref{58}, simple calculations give
\begin{align}
\frac{\dif}{\dif t}\int_\mathbb{R}a^{-X}(\xi)\eta(U(t,\xi)|\widetilde{U}(t,\xi))\dif\xi
=\dot{X}(t)Y(U)+\sum\limits_{i=1}\limits^{8}I_i,
\end{align}
where
\[
I_1:=\int_\mathbb{R}a^{-X}(u-\widetilde{u})\sigma(u-\widetilde{u})_{\xi}\dif{\xi},
\quad
I_2:=\int_\mathbb{R}a^{-X}[-p(v)\sigma v_{\xi}+p^{\prime}(\widetilde{v})\sigma\widetilde{v}_{\xi}(v-\widetilde{v})+p(\widetilde{v})\sigma v_{\xi}]\dif{\xi},
\]
\[
I_3:=\int_\mathbb{R}a^{-X}\frac{\Pi-\widetilde{\Pi}}{\mu}\sigma\tau(\Pi-\widetilde{\Pi})_{\xi}\dif{\xi},
\]
\[I_4:=\int_\mathbb{R}a^{-X}[-(u-\widetilde{u})(p(v)-p(\widetilde{v}))_{\xi}-p(v)u_{\xi}
+p^{\prime}(\widetilde{v})\widetilde{u}_{\xi}(v-\widetilde{v})+p(\widetilde{v})u_{\xi}]\dif{\xi},
\]
\[
I_5:=-\int_\mathbb{R}a^{-X}\frac{\Pi-\widetilde{\Pi}}{\mu}(v\Pi-\widetilde{v}\widetilde{\Pi})\dif{\xi},
\quad
I_6:=\int_\mathbb{R}a^{-X}[(u-\widetilde{u})(\Pi-\widetilde{\Pi})_{\xi}
+(u-\widetilde{u})_{\xi}(\Pi-\widetilde{\Pi})]\dif{\xi},
\]
\[
I_7:=-\int_\mathbb{R}a^{-X}(u-\widetilde{u})F_1\dif{\xi},
\quad
I_8:=-\int_\mathbb{R}a^{-X}\frac{\Pi-\widetilde{\Pi}}{\mu}F_2\dif{\xi},
\]
Integrating by parts, we have
\[
I_1=\int_\mathbb{R}a^{-X}(u-\widetilde{u})\sigma(u-\widetilde{u})_{\xi}\dif{\xi}
=-\frac{\sigma}{2}\int_\mathbb{R}a^{-X}_{\xi}(u-\widetilde{u})^2\dif{\xi},
\]
and
\[
\begin{aligned}
I_2&=\int_\mathbb{R}a^{-X}[-p(v)\sigma v_{\xi}+p^{\prime}(\widetilde{v})\sigma\widetilde{v}_{\xi}(v-\widetilde{v})+p(\widetilde{v})\sigma v_{\xi}]\dif{\xi}\\
&=\sigma\int_\mathbb{R}a^{-X}[-p(v)v_{\xi}+p(\widetilde{v})v_{\xi}]\dif{\xi}
-\sigma\int_\mathbb{R}a^{-X}p(\widetilde{v})(v-\widetilde{v})_{\xi}\dif\xi-
\sigma\int_\mathbb{R}a^{-X}_{\xi}p(\widetilde{v})(v-\widetilde{v})\dif\xi\\
&=\sigma\int_\mathbb{R}a^{-X}[-p(v)v_{\xi}+p(\widetilde{v})\widetilde{v}_{\xi}]\dif{\xi}
-\sigma\int_\mathbb{R}a^{-X}_{\xi}p(\widetilde{v})(v-\widetilde{v})\dif\xi\\
&=\sigma\int_\mathbb{R}a^{-X}[H(v)_{\xi}-H(\widetilde{v})_{\xi}]\dif{\xi}
-\sigma\int_\mathbb{R}a^{-X}_{\xi}p(\widetilde{v})(v-\widetilde{v})\dif\xi\\
&=-\sigma\int_\mathbb{R}a^{-X}_{\xi}(H(v)-H(\widetilde{v}))\dif{\xi}
+\sigma\int_\mathbb{R}a^{-X}_{\xi}H^{\prime}(\widetilde{v})(v-\widetilde{v})\dif\xi\\
&=-\sigma\int_\mathbb{R}a^{-X}_{\xi}H(v|\widetilde{v})\dif{\xi}.
\end{aligned}
\]
Similarly, we have
\[
I_3=\int_\mathbb{R}a^{-X}\frac{\Pi-\widetilde{\Pi}}{\mu}\sigma\tau(\Pi-\widetilde{\Pi})_{\xi}\dif{\xi}
=-\frac{\sigma\tau}{2\mu}\int_\mathbb{R}a^{-X}_{\xi}(\Pi-\widetilde{\Pi})^2\dif\xi,
\]
and
\[
\begin{aligned}
I_4&=\int_\mathbb{R}a^{-X}[-(u-\widetilde{u})(p(v)-p(\widetilde{v}))_{\xi}-p(v)u_{\xi}
+p^{\prime}(\widetilde{v})\widetilde{u}_{\xi}(v-\widetilde{v})+p(\widetilde{v})u_{\xi}]\dif{\xi},\\
&=\int_\mathbb{R}a^{-X}_{\xi}(u-\widetilde{u})(p(v)-p(\widetilde{v}))\dif\xi
+\int_\mathbb{R}a^{-X}(u-\widetilde{u})_{\xi}(p(v)-p(\widetilde{v}))\dif\xi\\
&+\int_\mathbb{R}a^{-X}[-p(v)+p(\widetilde{v})]u_{\xi}\dif\xi
+\int_\mathbb{R}a^{-X}p^{\prime}(\widetilde{v})\widetilde{u}_{\xi}(v-\widetilde{v})+u_{\xi}]\dif{\xi}\\
&+\int_\mathbb{R}a^{-X}p^{\prime}(\widetilde{v})\widetilde{u}_{\xi}(v-\widetilde{v})\dif{\xi}\\
&=\underbrace{\int_\mathbb{R}a^{-X}_{\xi}(u-\widetilde{u})(p(v)-p(\widetilde{v}))\dif\xi}\limits_{=:I_{9}}
-\int_\mathbb{R}a^{-X}(\widetilde{u}^R_{\xi}+\sigma\widetilde{v}^S_{\xi})p(v|\widetilde{v})\dif\xi.
\end{aligned}
\]
For estimates of $I_5, I_6$, noting that $v\Pi-\widetilde{v}\widetilde{\Pi}=v(\Pi-\widetilde{\Pi})+\widetilde{\Pi}(v-\widetilde{v})$,
we have
\[
\begin{aligned}
    I_5&=-\int_\mathbb{R}a^{-X}\frac{\Pi-\widetilde{\Pi}}{\mu}(v\Pi-\widetilde{v}\widetilde{\Pi})\dif{\xi}\\
       &=-\int_\mathbb{R}a^{-X}\frac{v}{\mu}(\Pi-\widetilde{\Pi})^2\dif{\xi}
       -\int_\mathbb{R}a^{-X}\frac{\widetilde{\Pi}}{\mu}(\Pi-\widetilde{\Pi})(v-\widetilde{v})\dif\xi,
   \end{aligned}
\]
and
\[
    I_6=\int_\mathbb{R}a^{-X}[(u-\widetilde{u})(\Pi-\widetilde{\Pi})_{\xi}+(u-\widetilde{u})_{\xi}(\Pi-\widetilde{\Pi})]\dif{\xi}
       =\underbrace{-\int_\mathbb{R}a^{-X}_{\xi}(u-\widetilde{u})(\Pi-\widetilde{\Pi})\dif{\xi}}\limits_{=:I_{10}}.
\]
Applying the equality $\alpha x^2+\beta xy+\kappa xz=\alpha(x+\frac{\beta}{2\alpha}y+\frac{\kappa}{2\alpha}z)^2-\frac{\beta^2}{4\alpha}y^2-\frac{\kappa}{4\alpha}z^2-\frac{\beta\kappa}{4\alpha}yz$ with
$x:=u-\widetilde{u},y:=p(v)-p(\widetilde{v}),z:=\Pi-\widetilde{\Pi}$, we have
\begin{align*}
I_1+I_{9}+I_{10}
                 =-&\frac{\sigma}{2}\int_\mathbb{R}a_{\xi}^{-X}\left(u-\widetilde{u}-\frac{p(v)-p(\widetilde{v})}{\sigma}
                 +\frac{\Pi-\widetilde{\Pi}}{\sigma}\right)^2\dif\xi
                 +\frac{1}{2\sigma}\int_\mathbb{R}a_{\xi}^{-X}((p(v)-p(\widetilde{v}))^2\dif\xi\\
                 &+\frac{1}{2\sigma}\int_\mathbb{R}a^{-X}_{\xi}(\Pi-\widetilde{\Pi})^2\dif\xi
                 -\frac{1}{\sigma}\int_\mathbb{R}a^{-X}_{\xi}(\Pi-\widetilde{\Pi})(p(v)-p(\widetilde{v}))\dif\xi.
\end{align*}
Therefore, combining the above estimates, we get the desired result.
\end{proof}
Now, using Lemma \ref{Le6} with a change of variable $\xi\mapsto\xi+X(t)$, we have
\begin{equation}\label{62}
\frac{\dif}{\dif t}\int_\mathbb{R}a\eta(U^{X}|\widetilde{U}^{X})d\xi=\dot{X}(t)Y(U^{X})+J^{bad}(U^{X})-J^{good}(U^{X})
\end{equation}
where
\[
\widetilde{U}^{X}:=
\begin{pmatrix}
\widetilde{v}^{X}\\
\widetilde{u}^{X}\\
\widetilde{\Pi}^{X}
\end{pmatrix}
=
\begin{pmatrix}
(\widetilde{v}^R)^{X}+\widetilde{v}^S-v_m\\
(\widetilde{u}^R)^{X}+\widetilde{u}^S-u_m\\
\widetilde{\Pi}^S+\left(\mu\frac{\widetilde{u}^R_{\xi}}{\widetilde{v}^R}\right)^X
\end{pmatrix},
U^{X}:=
\begin{pmatrix}
v^{X}\\
u^{X}\\
\Pi^{X}
\end{pmatrix},
\]
and $((\widetilde{v}^R)^{X}, (\widetilde{u}^R)^{X})=(\widetilde{v}^R, \widetilde{u}^R)\left(t,x+X(t)\right)$,
$(v^{X}, u^{X}, \Pi^{X})=(v, u, \Pi)(t,\xi+X(t))$.

For the terms $J^{bad}$ and $J^{good}$, we use the following notations:
\begin{equation}\label{63}
  \begin{aligned}
  J^{bad}(U):=\sum\limits_{i=1}\limits^5B_i(U)+S_1(U)+S_1(U),\quad
  J^{good}(U):=\sum\limits_{i=1}\limits^3G_i(U)+G(U)+G^R(U),
  \end{aligned}
\end{equation}
where
\[
  B_1(U):=\frac{1}{2\sigma}\int_\mathbb{R}a_{\xi}((p(v)-p(\widetilde{v}^{X}))^2\dif\xi,
  B_2(U):=\sigma\int_\mathbb{R}a(\widetilde{v}^S)_{\xi}p(v|\widetilde{v}^{X})\dif\xi,
\]
\[
  B_3(U):=-\frac{1}{\sigma}\int_\mathbb{R}a_{\xi}(\Pi-\widetilde{\Pi}^{X})(p(v)-p(\widetilde{v}^{X}))\dif\xi,
\quad  B_4(U):=-\int_\mathbb{R}a\frac{\widetilde{\Pi}^{X}}{\mu}(\Pi-\widetilde{\Pi}^{X})(v-\widetilde{v}^{X})\dif\xi,
\]
\[
  B_5(U):=\frac{1}{2\sigma}\int_\mathbb{R}a_{\xi}(\Pi-\widetilde{\Pi}^{X})^2\dif\xi,
\quad  S_1(U):=-\int_\mathbb{R}a(u-\widetilde{u}^{X})F_1^X\dif\xi,
\]
\[
  S_2(U):=-\int_\mathbb{R}a\frac{(\Pi-\widetilde{\Pi}^{X})}{\mu}F_2^{X}\dif\xi,
\]
and
\[
  G_1(U):=\frac{\sigma}{2}\int_\mathbb{R}a_{\xi}\left(u-\widetilde{u}^{X}-\frac{p(v)-p(\widetilde{v}^{X})}{\sigma}
  +\frac{\Pi-\widetilde{\Pi}^{X}}{\sigma}\right)^2\dif\xi,
\]
\[
  G_2(U):=\sigma\int_\mathbb{R}a_{\xi}H(v|\widetilde{v}^{X})\dif\xi,
 \quad G_3(U):=\frac{\sigma\tau}{2\mu}\int_\mathbb{R}a_{\xi}(\Pi-\widetilde{\Pi}^{X})^2\dif\xi,
\]
\[
  G(U):=\int_\mathbb{R}a\frac{v}{\mu}(\Pi-\widetilde{\Pi}^{X})^2\dif\xi,
 \quad G^R(U):=\int_\mathbb{R}a(\widetilde{u}^R_{\xi})^{X}p(v|\widetilde{v}^{X})\dif\xi.
\]
Without of confusion, we use the notation $(v,u,\Pi)=(v^X,u^X,\Pi^X)$ for simplicity.

For $Y$, we have from \eqref{58} that
\[
\begin{aligned}
&Y(U)=-\int_\mathbb{R}a_{\xi}\eta(U|\widetilde{U}^{X})\dif\xi
+\int_\mathbb{R}a\left(\widetilde{u}^S_{\xi}(u-\widetilde{u}^{X})
-p^{\prime}(\widetilde{v}^X)\widetilde{v}^S_{\xi}(v-\widetilde{v}^{X})
+\frac{\tau}{\mu}\widetilde{\Pi}^S_{\xi}(\Pi-\widetilde{\Pi}^{X})\right)\dif\xi\\
=&-\int_\mathbb{R}a_{\xi}\left(\frac{|u-\widetilde{u}^{X}|^2}{2}+H(v|\widetilde{v}^{X})
+\frac{\tau|\Pi-\widetilde{\Pi}^{X}|^2}{2\mu}\right)\dif\xi
+\int_\mathbb{R}a\left(\widetilde{u}^S_{\xi}(u-\widetilde{u}^{X})
-p^{\prime}(\widetilde{v}^X)\widetilde{v}^S_{\xi}(v-\widetilde{v}^{X})\right)\dif\xi\\
&+\int_\mathbb{R}a\frac{\tau}{\mu}\widetilde{\Pi}_{\xi}(\Pi-\widetilde{\Pi}^{X})\dif\xi.
\end{aligned}
\]
We rewrite the  function $Y$ as follows:
\[
Y:=\sum\limits_{i=1}\limits^8Y_i,
\]
where
\[
  Y_1(U):=\int_\mathbb{R}a\frac{\widetilde{u}^S_{\xi}}{\sigma}(p(v)-p(\widetilde{v}^X))\dif\xi,
  \quad Y_2(U):=-\int_\mathbb{R}ap(\widetilde{v}^S)_{\xi}(v-\widetilde{v}^{X})\dif\xi,
\]
\[
  Y_3(U):=\int_\mathbb{R}a\widetilde{u}^S_{\xi}\left(u-\widetilde{u}^X-\frac{p(v)-p(\widetilde{v}^X)}
  {\sigma}\right)\dif\xi,
  Y_4(U)=-\int_\mathbb{R}a\left(p^{\prime}(\widetilde{v}^X)\widetilde{v}^S_{\xi}
  -p(\widetilde{v}^S)_{\xi}\right)(v-\widetilde{v}^{X})\dif\xi,
\]
\[
  Y_5(U):=\int_\mathbb{R}a\frac{\tau}{\mu}(\widetilde{\Pi}^S)_{\xi}(\Pi-\widetilde{\Pi}^{X})\dif\xi,
  \quad Y_6(U):=-\int_\mathbb{R}a_{\xi}\frac{\tau|\Pi-\widetilde{\Pi}^{X}|^2}{2\mu}\dif\xi,
\]
\[
\begin{aligned}
  Y_7(U):=-\frac{1}{2}\int_\mathbb{R}a_{\xi}\left(u-\widetilde{u}^X-\frac{p(v)-p(\widetilde{v}^X)}{\sigma}\right)
          \cdot\left(u-\widetilde{u}^X+\frac{p(v)-p(\widetilde{v}^X)}{\sigma}\right)\dif\xi,
\end{aligned}
\]
\[
Y_8(U):=-\int_\mathbb{R}a_{\xi}H(v|\widetilde{v}^{X})\dif\xi
-\frac{1}{2}\int_\mathbb{R}a_{\xi}(\frac{p(v)-p(\widetilde{v}^X)}{\sigma})^2\dif\xi.
\]
Notice from \eqref{eq35} that
\begin{equation}\label{64}
  \dot{X}(t)=-\frac{M}{\delta_S}(Y_1+Y_2),
\end{equation}
which implies
\begin{equation}\label{65}
  \dot{X}(t)Y=-\frac{\delta_S}{M}|\dot X(t)|^2+\dot X(t)\sum\limits_{i=3}\limits^8Y_i.
\end{equation}
The following lemma is given by \cite{WY}.
\begin{lemma}\label{Le8}
There exist a constant $C>0$ such that
\[
\begin{aligned}
  &-\frac{\delta_S}{2M}|\dot{X}|^2+B_1+B_2-G_2-\frac{3}{4}D\\
  &\leq
  C\int_\mathbb{R}\left(-\widetilde{v}^S_{\xi}|p(v)-p(\widetilde{v}^X)|^2
  +|a_{\xi}||p(v)-p(\widetilde{v}^X)|^3
  +|a_{\xi}||(\widetilde{v}^R)^X-v_m||p(v)-p(\widetilde{v}^X)|^2\right)\dif\xi. 
\end{aligned}
\]
where
\begin{align*}
D=\int_{\mathbb R} \frac{a}{\gamma p(v)} |\partial_\xi(p(v)-p(\tilde v^X))|^2 \dif \xi.
\end{align*}
\end{lemma}
We note that although the definition of $X(t)$ is different with that in \cite{WY}, the proof of the above lemma is almost the same.

Now, we are ready to prove Lemma \ref{Le5}.

{\bf{Proof of Lemma \ref{Le5}}}:

 First of all, we use \eqref{62}, \eqref{63} and \eqref{65} to have
\[
\begin{aligned}
  \frac{\dif}{\dif t}\int_\mathbb{R}a&\eta(U^{X}|\widetilde{U}^{X})\dif\xi
  =-\frac{\delta_S}{2M}|\dot{X}|^2+B_1+B_2-G_2-\frac{3}{4}D\\
  &+\left(-\frac{\delta_S}{2M}|\dot{X}|^2\right)+\dot{X}\sum\limits_{i=3}^8Y_i
  +\sum\limits_{i=3}^5B_i+S_1+S_2-G_1-G_3-G-G^R+\frac{3}{4}D.
\end{aligned}
\]
Using Lemma \ref{Le8} and Young's inequality, we find that there exist $C>0$ such that
\[
\begin{aligned}
  &\frac{\dif}{\dif t}\int_\mathbb{R}a\eta(U^{X}|\widetilde{U}^{X})\dif\xi\\
  &\leq-C\int_\mathbb{R}(\widetilde{v}^S)_{\xi}|p(v)-p(\widetilde{v}^X)|^2\dif\xi
  +\underbrace{C\int_\mathbb{R}|a_{\xi}||p(v)-p(\widetilde{v}^X)|^3\dif{\xi}}\limits_{=:K_1}\\
  &+\underbrace{C\int_\mathbb{R}|a_{\xi}||(\widetilde{v}^R)^X-v_m||p(v)-p(\widetilde{v}^X)|^2\dif{\xi}}\limits_{=:K_2}\\
  &+\left(-\frac{\delta_S}{4M}|\dot{X}|^2\right)
  +\frac{C}{\delta_S}\sum\limits_{i=3}^8|Y_i|^2+\sum\limits_{i=3}^5B_i+S_1+S_2-G_1-G_3-G-G^R+D.
\end{aligned}
\]
Next, we set
\begin{equation}\label{78}
G^S:=\int_\mathbb{R}(\widetilde{v}^S)_{\xi}|p(v)-p(\widetilde{v}^X)|^2\dif \xi.
\end{equation}
For $K_1$, we use the notation $w=p(v)-p(\widetilde{v}^X)$. Use \eqref{56}, the interpolation inequality and Young's inequality, one obtain
\begin{align*}
  K_1&\leq C\frac{\lambda}{\delta_S}\int_\mathbb{R}\|w\|^2_{L^{\infty}}|(\widetilde{v}^S)_{\xi}||w|\dif\xi\\
  &\leq C\frac{\lambda}{\delta_S}\|w_{\xi}\|_{L^2}\|w\|_{L^2}\sqrt{\int_\mathbb{R}|(\widetilde{v}^S)_{\xi}|w^2\dif\xi}
  \sqrt{\int_\mathbb{R}|(\widetilde{v}^S)_{\xi}|\dif\xi}\\
&\leq C\varepsilon_1\|w_{\xi}\|_{L^2}\sqrt{\int_\mathbb{R}|(\widetilde{v}^S)_{\xi}|w^2\dif\xi}\\
  &\leq C\varepsilon_1(D+G^S).
\end{align*}
Similarly, for $K_2$,  using \eqref{45}, Lemma \ref{Le9} and Young's inequality,  we have
\begin{align*}
 K_2&\leq C\frac{\lambda}{\delta_S}\|w\|^2_{L^4}|\left\||(\widetilde{v}^S)_{\xi}||(\widetilde{v}^R)^X-v_m|\right\|_{L^2}\\
  &\leq C\frac{\lambda}{\delta_S}\|w_{\xi}\|^{1/2}_{L^2}\|w\|^{3/2}_{L^2}
  \||(\widetilde{v}^S)_{\xi}||(\widetilde{v}^R)^X-v_m|\|_{L^2}\\
 &\leq C\frac{\varepsilon_1}{\sqrt{\delta_S}}\|w_{\xi}\|^{1/2}_{L^2}\left\||(\widetilde{v}^S)_{\xi}||(\widetilde{v}^R)^X-v_m|\right\|_{L^2}\\
 &\leq C\varepsilon_1\left(D+\frac{1}{\delta_S}\left\||(\widetilde{v}^S)_{\xi}||(\widetilde{v}^R)^X-v_m|\right\|_{L^2}^2\right).
\end{align*}
For $Y_i$, by using H\"{o}lder inequality, we first get
\[
\begin{aligned}
 |Y_3|&\leq C\frac{\delta_S}{\lambda}\int_\mathbb{R}|a_{\xi}|
 \left|u-\widetilde{u}^{X}-\frac{p(v)-p(\widetilde{v}^{X})}{\sigma}\right|\dif\xi\leq C\frac{\delta_S}{\sqrt{\lambda}}\sqrt{G_1+G}.
\end{aligned}
\]
Meanwhile, using Lemmas \ref{Le1}, \ref{Le2}, \ref{Le4} and H\"{o}lder inequality, we have
\[
|Y_4|\leq C\int_\mathbb{R}|(\widetilde{v}^{R})^X-v_m|\widetilde{v}^{S}_{\xi}|w|\dif\xi
\leq C\delta_R\sqrt{\delta_S}\sqrt{G^S}.
\]
Similarly, for $Y_5, Y_6, Y_7$, we have
\begin{align*}
 |Y_5|\leq C\tau\int_\mathbb{R}|(\widetilde{\Pi}^{S})_{\xi}||\Pi-\widetilde{\Pi}^{X}|\dif\xi\leq C\tau\delta_S^3\sqrt{G},
\end{align*}
\[
|Y_6|\leq \sqrt{\tau} |\Pi-\widetilde \Pi^X|_{L^\infty} \int_\mathbb {R} a_\xi \frac{\sqrt{\tau} |\Pi-\widetilde \Pi|}{2\mu}\dif x \le C \varepsilon_1 \sqrt{\tau} \lambda\sqrt{\delta_S} \sqrt{G},
\]
and
\[
 |Y_7|\leq C\varepsilon_1|G_1+G|^{\frac{1}{2}}\|a_{\xi}\|_{L^{\infty}(\mathbb{R})}^{\frac{1}{2}}\leq C\varepsilon_1\lambda(\delta_S)^{\frac{1}{2}}|G_1+G|^{\frac{1}{2}}.
\]
For $Y_8$, using Lemma \ref{Le1}, we have
\[
Y_8\leq C\int_\mathbb{R}|a_{\xi}|w^2\dif\xi.
\]
Then, using \eqref{56} and Lemma \ref{Le2}, we have
\[
\begin{aligned}
 \frac{C}{\delta_S}|Y_8|^2\leq \frac{C\lambda^2}{\delta_S^3}\left(\int_\mathbb{R}|\widetilde{v}^S_{\xi}|w^2\dif\xi\right)^2
 \leq \frac{C\lambda^2}{\delta_S}\|w\|^2_{L^2(\mathbb{R})}\int_\mathbb{R}|\widetilde{v}^S_{\xi}|w^2\dif\xi
 \leq C\varepsilon_1^2G^S.
\end{aligned}
\]
For $B_3$, using \eqref{56}, Lemma \ref{Le2} and Young's inequality, we have
\[
\begin{aligned}
   |B_3|&\leq C\lambda\int_\mathbb{R}(\Pi-\widetilde{\Pi}^X)|(\widetilde{v}^{S})_{\xi}|^{\frac{1}{2}}w\dif\xi
   \leq C\lambda(G+G^S).
\end{aligned}
\]
For $B_4$,  we have
\[
\begin{aligned}
   |B_4|&\leq C\int_\mathbb{R}(|\widetilde{\Pi}^S|+|(\widetilde{u}^R)^X_{\xi}|)|\Pi-\widetilde{\Pi}^X||v-\widetilde{v}^X|\dif\xi\\
        &\leq C\int_\mathbb{R}|\widetilde{\Pi}^S||\Pi-\widetilde{\Pi}^X||v-\widetilde{v}^X|\dif\xi
        +C\int_\mathbb{R}(\widetilde{u}^R)^X_{\xi}|\Pi-\widetilde{\Pi}^X||v-\widetilde{v}^X|\dif\xi\\
        &:=J_1+J_2.
\end{aligned}
\]
Using $\widetilde{S}^S\sim\widetilde{v}^S_{\xi}$, Lemma \ref{Le2} and Young's inequality, we have
\[
   J_1\leq C\delta_S\int_\mathbb{R}|\Pi-\widetilde{\Pi}^X||\widetilde{v}^S_{\xi}|^{\frac{1}{2}}|v-\widetilde{v}^X|\dif\xi
        \leq C\delta_S(G+G^S).
\]
Similarly ,for $J_2$, using Lemma \ref{Le4} and Young's inequality, we have
\[
\begin{aligned}
  J_2\leq C\delta_R^{\frac{1}{2}}\int_\mathbb{R}|\Pi-\widetilde{\Pi}^X||(\widetilde{u}^R)^X_{\xi}|^{\frac{1}{2}}|v-\widetilde{v}^X|\dif\xi
  \leq C \delta_R^{\frac{1}{2}}(G+G^R).
\end{aligned}
\]
Thus,
\[
|B_4|\leq C(\delta_S+\delta_R^{\frac{1}{2}})G+\delta_SG^S+\delta_R^{\frac{1}{2}}G^R.
\]
For $B_5$, using $\|a_{\xi}\|_{L^{\infty}}\leq C\lambda\delta_S$, we have
\[
|B_5|\leq C\lambda\delta_SG.
\]

Then, using \eqref{44}, Sobolev's imbedding theorem and Young's inequality,  and noting that
\begin{equation}\label{79}
\begin{aligned}
 &|p(\widetilde{v}^X)_{\xi}-p((\widetilde{v}^R)^X)_{\xi}-p(\widetilde{v}^S)_{\xi}|\leq C(|(\widetilde{v}^R)^X_{\xi}||\widetilde{v}^S-v_m|+|\widetilde{v}^S_{\xi}||(\widetilde{v}^R)^X-v_m|),
\end{aligned}
\end{equation}
and
\[
\mu\left(\frac{(\widetilde{u}^R)^X_{\xi}}{(\widetilde{v}^R)^X}\right)_{\xi}=
\mu\left(\frac{(\widetilde{u}^R)^X_{\xi\xi}(\widetilde{v}^R)^X-(\widetilde{u}^R)^X_{\xi}
(\widetilde{v}^R)^X_{\xi}}{((\widetilde{v}^R)^X)^2}\right),
\]
we have
\[
\begin{aligned}
& |S_1|\leq C\int_\mathbb{R}|u-\widetilde{u}^X|\left(|p(\widetilde{v}^X)_{\xi}-p((\widetilde{v}^R)^X)_{\xi}-p(\widetilde{v}^S)_{\xi}|
 +|(\widetilde{u}^R)^X_{\xi\xi}|+|(\widetilde{u}^R)^X_{\xi}(\widetilde{v}^R)^X_{\xi}|\right)\dif\xi\\
 &\leq C\varepsilon_1\left(\left\|(\widetilde{v}^R)^X_{\xi}||\widetilde{v}^S-v_m|\right\|_{L^2}
 +\left\||\widetilde{v}^S_{\xi}||(\widetilde{v}^R)^X-v_m\right\|_{L^2}+
 \|(\widetilde{u}^R)^X_{\xi\xi}\|_{L^2}
 +\|(\widetilde{u}^R)^X_{\xi}\|_{L^4}\|(\widetilde{v}^R)^X_{\xi}\|_{L^4}\right).
\end{aligned}
\]
Similarly, for $S_2$, using \eqref{44}, $|\widetilde{\Pi}^S|\sim|\widetilde{v}^S_{\xi}|$ and Lemma \ref{Le4}, we have
\[
\begin{aligned}
 |S_2|&\leq C\int_\mathbb{R}|\Pi-\widetilde{\Pi}^X|\left[\tau(|(\widetilde{u}^R)^X_{t\xi}|+|(\widetilde{u}^R)^X_{\xi}(\widetilde{v}^R)^X_{t}|
 +|(\widetilde{u}^R)^X_{\xi\xi}|+|(\widetilde{u}^R)^X_{\xi}(\widetilde{v}^R)^X_{\xi}|)\right.\\
 &\qquad\left.+|(\widetilde{v}^R)^X_{\xi}||\widetilde{v}^S-v_m|
 +|\widetilde{v}^S_{\xi}||(\widetilde{v}^R)^X-v_m|\right]\dif\xi\\
       &\leq \frac{1}{64}G+C\tau^2\left(\left\|\left((\widetilde{v}^R)^X_{\xi\xi},(\widetilde{v}^R)^X_{\xi\xi}\right)\right\|^2_{L^2}
 +\|(\widetilde{u}^R)^X_{\xi}\|^4_{L^4}
 +\|(\widetilde{u}^R)^X_{\xi}\|^2_{L^4}\|(\widetilde{v}^R)^X_{\xi}\|^2_{L^4}\right)\\
 &\qquad+C\left(\left\|(\widetilde{v}^R)^X_{\xi}||\widetilde{v}^S-v_m|\right\|^2_{L^2}
 +\left\||\widetilde{v}^S_{\xi}||(\widetilde{v}^R)^X-v_m\right\|^2_{L^2}\right).
\end{aligned}
\]
Noting that from \eqref{eq33}, we get $\widetilde{u}^R_x=\sqrt{-p^{\prime}(\widetilde{v}^R)}\widetilde{v}^R_x$, then we have $\widetilde{u}^R_x\sim\widetilde{v}^R_x$.\\
In addition, since
\[
(p(v)-p(\widetilde{v}^X))_{\xi}=p^{\prime}(v)(v-\widetilde{v}^X)_{\xi}
+\widetilde{v}^X_{\xi}(p^{\prime}(v)-p^{\prime}(\widetilde{v}^X)),
\]
using \eqref{eq41}, Lemma \ref{Le2} and Sobolev's imbedding theorem, we have
\[
\begin{aligned}
D&\leq C\left(\int_\mathbb{R}\left(\partial_{\xi}(v-\widetilde{v}^X)\right)^2\dif\xi+
\int_\mathbb{R}[(\widetilde{v}^S)^2+(\widetilde{v}^R)^2](v-\widetilde{v}^X)^2\dif\xi\right)\\
&\leq C\left(\int_\mathbb{R}\left(\partial_{\xi}(v-\widetilde{v}^X)\right)^2\dif\xi+
\delta_S^2G^S+\delta_RG^R\right).
\end{aligned}
\]
Finally, let $\epsilon=\delta_R^3$, using the smallness of $\delta_S,\delta_R,\varepsilon_1$ and combining the above estimates, then integrating it over $[0,T]$ for any $t\leq T$, we derive that
\[
\begin{aligned}
&\int_\mathbb{R}\eta(U|\widetilde{U}^X)\dif\xi+\delta_S\int_0^t|\dot{X}|^2\dif t+\int_0^t(G_1+G_2+G_3+G+G^R+G^S)\dif t\\
&\quad \leq C\int_\mathbb{R}\eta(U(0,\xi)|\widetilde{U}^X(0,\xi))\dif\xi
+C\int_0^t\int_\mathbb{R}\left(\partial_{\xi}(v-\widetilde{v}^X)\right)^2\dif\xi\dif t+C\delta_R^{\theta},
\end{aligned}
\]
where $\theta=\min\{\frac{1}{2}, \frac{3}{2}-\frac{1}{q}\}$.

Using Lemma \ref{Le1}, we have finished proof of Lemma \ref{Le5}.

\subsection{High-order energy estimates}
In this section, we show the high-order energy estimates for the system \eqref{50}. Let $\Phi=v-\widetilde{v},\Psi=u-\widetilde{u},Q=\Pi-\widetilde{\Pi}$, then we have the following lemma.
\begin{lemma}\label{Le10}
There exist a constant $C>0$ such that for $0\leq t\leq T$, we have
\[
\begin{aligned}
&\|\partial_{\xi}\Phi\|_{H^1}^2+\|\partial_{\xi}\Psi\|_{H^1}^2
+\tau\|\partial_{\xi}Q\|_{H^1}^2
+\int_0^t\|\partial_{\xi}Q\|_{H^1}^2\dif t\leq
 C\delta_S^2\int_0^t|\dot{X}|^2\dif t\\
&+C(\|\partial_{\xi}\Phi_0\|_{H^1}^2+\|\partial_{\xi}\Psi_0\|_{H^1}^2
+\tau\|\partial_{\xi}Q_0\|_{H^1}^2)
+C\int_0^t\left(\delta_S\mathcal{G}^S(U)+
\delta_R^{\frac{1}{2}}\mathcal{G}^R(U)\right)\dif t\\
&+C(\delta+\varepsilon_1)\int_0^t(\|\partial_{\xi}\Phi\|_{H^1}^2
+\|\partial_{\xi}\Psi\|_{H^1}^2)\dif t
+C\delta \int_0^tG(U)\dif t+C\delta_R^{\theta},
\end{aligned}
\]
where $\delta=\delta_S+\delta_R$,
$\Phi_0=v_0-\widetilde{v}_0(\xi), \Psi_0=u_0-\widetilde{u}_0(\xi), Q_0=\Pi_0-\widetilde{\Pi}_0(\xi)$
and $\theta=\min\{\frac{1}{2}, \frac{3}{2}-\frac{1}{q}\}$.
\end{lemma}
\begin{proof}
Applying $\partial^k_{\xi}(k=1,2)$ to the system \eqref{50}, we derive that
\begin{equation}\label{80}
\begin{cases}
\partial_{t}\partial_{\xi}^k\Phi-\sigma\partial_{\xi}^{k+1}\Phi-\dot{X}(t)\partial_{\xi}^{k+1}(\widetilde{v}^S)^{-X}-\partial_{\xi}^{k+1}\Psi=0,\\
\partial_{t}\partial_{\xi}^k\Psi-\sigma\partial_{\xi}^{k+1}\Psi-\dot{X}(t)\partial_{\xi}^{k+1}(\widetilde{u}^S)^{-X}
          +p^{\prime}(v)\partial_{\xi}^{k+1}\Phi=\partial_{\xi}^{k+1}Q+F_3^k,\\
\tau\partial_{t}\partial_{\xi}^kQ-\sigma\tau\partial_{\xi}^{k+1}Q-\tau\dot{X}(t)\partial_{\xi}^{k+1}(\widetilde{\Pi}^S)^{-X}
             +v\partial_{\xi}^{k}Q=\mu\partial_{\xi}^{k+1}\Psi+F_4^k,
\end{cases}
\end{equation}
where $$F_3^k=p^{\prime}(v)\partial^{k+1}_{\xi}(v-\widetilde{v})-\partial^{k+1}_{\xi}(p(v)-p(\widetilde{v}))-\partial^k_{\xi}F_1,$$
and
$$F_4^k=v\partial^k_{\xi}(\Pi-\widetilde{\Pi})-\partial^k_{\xi}(v\Pi-\widetilde{v}\widetilde{\Pi})-\partial^k_{\xi}F_2.$$

Multiplying the above equations by $-p^{\prime}(v)\partial_{\xi}^k\Phi,\partial_{\xi}^k\Psi,\frac{1}{\mu}\partial_{\xi}^kQ$, respectively, and integrating over $\mathbb{R}$, we get
\begin{align}\label{new1}
  \frac{\dif}{\dif t}\int_\mathbb{R}\left(-\frac{p^{\prime}(v)}{2}(\partial_{\xi}^k\Phi)^2
  +\frac{1}{2}(\partial_{\xi}^k\Psi)^2
  +\frac{\tau}{2\mu}(\partial_{\xi}^kQ)^2\right)\dif\xi
   +\int_\mathbb{R}\frac{v}{\mu}(\partial_{\xi}^kQ)^2\dif\xi
    =:\sum\limits_{i=1}\limits^{8}R^k_i,
\end{align}
where
\[
\begin{aligned}
&R^k_1=-\int_\mathbb{R}\frac{p^{\prime\prime}(v)}{2}v_t(\partial_{\xi}^k\Phi)^2\dif\xi \quad
,R^k_2=\int_\mathbb{R}\frac{\sigma p^{\prime\prime}(v)}{2}v_{\xi}(\partial_{\xi}^k\Phi)^2\dif\xi,\\
&R^k_3=-\int_\mathbb{R}p^{\prime\prime}(v)v_{\xi}\partial_{\xi}^k\Phi\partial_{\xi}^k\Psi \dif\xi,\quad
R^k_4=-\int_\mathbb{R}\dot{X}(t)\partial_{\xi}^{k+1}(\widetilde{v}^S)^{-X}p^{\prime}(v)\partial_{\xi}^k\Phi \dif\xi,\\
&R^k_5=\int_\mathbb{R}\dot{X}(t)\partial_{\xi}^{k+1}(\widetilde{u}^S)^{-X}\partial_{\xi}^k\Psi \dif\xi,\quad
R^k_6=\int_\mathbb{R}\frac{\tau}{\mu}\dot{X}(t)\partial_{\xi}^{k+1}(\widetilde{\Pi}^S)^{-X}Q\dif\xi,\\
&R^k_7=\int_\mathbb{R}F_3^k\partial_{\xi}^k\Psi \dif\xi,\quad
R^k_8=\int_\mathbb{R}\frac{1}{\mu}F_4^k\partial_{\xi}^kQ\dif\xi.
\end{aligned}
\]
Firstly, by using \eqref{eq41}, Lemmas \ref{Le2} and \ref{Le4}, we have
\begin{equation}\label{81}
\big|(\widetilde{v}_{\xi},\widetilde{u}_{\xi},\widetilde{\Pi})\big|\leq C\delta.
\end{equation}
Using \eqref{eq40}, \eqref{45}, \eqref{81} and Sobolev's imbedding theorem, we have
\begin{align*}
&R^k_1+R^k_2=-\int_\mathbb{R}\frac{p^{\prime\prime}(v)}{2}(v_t-\sigma v_{\xi})(\partial_{\xi}^k\Phi)^2\dif\xi
  =-\int_\mathbb{R}\frac{p^{\prime\prime}(v)}{2}u_{\xi}(\partial_{\xi}^k\Phi)^2\dif\xi\\
 &=-\int_\mathbb{R}\frac{p^{\prime\prime}(v)}{2}(u-\widetilde{u})_{\xi}(\partial_{\xi}^k\Phi)^2\dif\xi
       -\int_\mathbb{R}\frac{p^{\prime\prime}(v)}{2}\widetilde{u}_{\xi}(\partial_{\xi}^k\Phi)^2\dif\xi\leq C(\varepsilon_1+\delta)\int_\mathbb{R}(\partial_{\xi}^k\Phi)^2\dif\xi.
\end{align*}

Similarly, for $R^k_3$, using \eqref{eq40}, \eqref{45}, \eqref{81}, Young's inequality and Sobolev's imbedding theorem, we have
\[
\begin{aligned}
  R_3^k&=-\int_\mathbb{R}p^{\prime\prime}(v)(v-\widetilde{v})_{\xi}\partial_{\xi}^k\Phi\partial_{\xi}^k\Psi \dif\xi
       -\int_\mathbb{R}p^{\prime\prime}(v)\widetilde{v}_{\xi}\partial_{\xi}^k\Phi\partial_{\xi}^k\Psi \dif\xi\\
       &\leq C(\varepsilon_1+\delta)\left(\int_\mathbb{R}(\partial_{\xi}^k\Phi)^2\dif\xi
       +\int_\mathbb{R}(\partial_{\xi}^k\Psi)^2\dif\xi\right).
\end{aligned}
\]
For $R_4^k$, using Lemma \ref{Le2} and Young's inequality, we have
\[
\begin{aligned}
R_4^k&\leq C\delta_S|\dot{X}|^2\int_\mathbb{R}|\partial_{\xi}^{k+1}(\widetilde{v}^S)^{-X}|\dif\xi
   +C\frac{1}{\delta_S}\int_\mathbb{R}|\partial_{\xi}^{k+1}(\widetilde{v}^S)^{-X}|(\partial_{\xi}^k\Phi)^2\dif\xi\\
   &\leq C\delta_S^2|\dot{X}|^2+C\delta_S\int_\mathbb{R}(\partial_{\xi}^k\Phi)^2\dif\xi.
\end{aligned}
\]
Similarly, for $R_5^k,R_6^k$, we have
\[
R_5^k\leq C\delta_S^2|\dot{X}|^2+C\delta_S\int_\mathbb{R}(\partial_{\xi}^k\Psi)^2\dif\xi,\quad
R_6^k\leq C\delta_S^2\tau|\dot{X}|^2+C\delta_S\tau\int_\mathbb{R}(\partial_{\xi}^kQ)^2\dif\xi.
\]
Then, we estimate $R_7^k$. For $k=1$, since
\[
\begin{aligned}
  &p^{\prime}(v)(v-\widetilde{v})_{\xi\xi}-(p(v)-p(\widetilde{v}))_{\xi\xi}\\
   &=-p^{\prime\prime}(v)v^2_{\xi}-p^{\prime}(v)\widetilde{v}_{\xi\xi}
                    +p^{\prime\prime}(\widetilde{v})\widetilde{v}^2_{\xi}
                    +p^{\prime}(\widetilde{v})\widetilde{v}_{\xi\xi}\\
   &=-p^{\prime\prime}(v)(v_{\xi}-\widetilde{v}_{\xi})^2
   -2p^{\prime\prime}(v)(v_{\xi}-\widetilde{v}_{\xi})\widetilde{v}_{\xi}
     -(p^{\prime\prime}(v)-p^{\prime\prime}(\widetilde{v}))\widetilde{v}^2_{\xi}
     -(p^{\prime}(v)-p^{\prime}(\widetilde{v}))\widetilde{v}_{\xi\xi},
\end{aligned}
\]
we have
\[
\begin{aligned}
  \int_\mathbb{R}&(u-\widetilde{u})_{\xi}(p^{\prime}(v)(v-\widetilde{v})_{\xi\xi}-(p(v)-p(\widetilde{v}))_{\xi\xi})\dif\xi\\
      &\leq C\int_\mathbb{R}|(u-\widetilde{u})_{\xi}||(v-\widetilde{v})_{\xi}|^2\dif\xi
      +C\int_\mathbb{R}|(u-\widetilde{u})_{\xi}||(v-\widetilde{v})_{\xi}||\widetilde{v}_{\xi}|\dif\xi\\
      &\quad +C\int_\mathbb{R}|(u-\widetilde{u})_{\xi}|
      |p^{\prime\prime}(v)-p^{\prime\prime}(\widetilde{v})||\widetilde{v}_{\xi}|^2\dif\xi
      +C\int_\mathbb{R}|(u-\widetilde{u})_{\xi}||p^{\prime}(v)-p^{\prime}(\widetilde{v})||\widetilde{v}_{\xi\xi}|\dif\xi.
\end{aligned}
\]
Using \eqref{45}, \eqref{81}, Young's inequality and Sobolev's imbedding theorem, we have
\[
\int_\mathbb{R}|(u-\widetilde{u})_{\xi}||(v-\widetilde{v})_{\xi}|^2\dif\xi\leq
 C\varepsilon_1\int_\mathbb{R}|\partial_{\xi}\Phi|^2\dif\xi.
\]
Similarly, using Lemma \ref{Le2}, Lemma \ref{Le4} and Young's inequality,, we have
\[
\begin{aligned}
\int_\mathbb{R}|(u-\widetilde{u})_{\xi}||(v-\widetilde{v})_{\xi}||\widetilde{v}_{\xi}|\dif\xi\leq C\delta(\int_\mathbb{R}|\partial_{\xi}\Phi|^2\dif\xi+\int_\mathbb{R}|\partial_{\xi}\Psi|^2\dif\xi),
\end{aligned}
\]
and
\begin{align*}
&\int_\mathbb{R}|(u-\widetilde{u})_{\xi}||p^{\prime\prime}(v)
-p^{\prime\prime}(\widetilde{v})||\widetilde{v}_{\xi}|^2\dif\xi\\
&\le C\int_\mathbb{R}|(u-\widetilde{u})_{\xi}|
|p^{\prime\prime}(v)-p^{\prime\prime}(\widetilde{v})||\widetilde{v}^R_{\xi}|^2\dif\xi
+C\int_\mathbb{R}|(u-\widetilde{u})_{\xi}||p^{\prime\prime}(v)-p^{\prime\prime}(\widetilde{v})|
|(\widetilde{v}^S)^{-X}_{\xi}|^2\dif\xi\\
&\le C\delta_R^{\frac{3}{2}}\int_\mathbb{R}|(u-\widetilde{u})_{\xi}| |v-\widetilde{v}||\widetilde{u}^R_{\xi}|^{\frac{1}{2}}\dif\xi+C\delta_S^3\int_\mathbb{R}|(u-\widetilde{u})_{\xi}||(\widetilde{v}^S)^{-X}_{\xi}|^{\frac{1}{2}}|v-\widetilde{v}|\dif\xi\\
&\le C\delta_R^{\frac{3}{2}}\left(\int_\mathbb{R}|\partial_{\xi}\Psi|^2\dif\xi+\mathcal{G}^R(U)\right)+C\delta_S^3\left(\int_\mathbb{R}(\partial_{\xi}\Psi)^2\dif\xi+\mathcal{G}^S(U)\right).
\end{align*}
In a similar way, using Lemma \ref{Le2}, Lemma \ref{Le4} and Young's inequality, we have
\begin{align*}
&\int_\mathbb{R}|(u-\widetilde{u})_{\xi}||p^{\prime}(v)-p^{\prime}(\widetilde{v})||\widetilde{v}_{\xi\xi}|\dif\xi\\
&\leq C\int_\mathbb{R}|(u-\widetilde{u})_{\xi}|
|p^{\prime}(v)-p^{\prime}(\widetilde{v})||\widetilde{v}^R_{\xi\xi}|\dif\xi
+C\int_\mathbb{R}|(u-\widetilde{u})_{\xi}||p^{\prime}(v)-p^{\prime}(\widetilde{v})|
|(\widetilde{v}^S)^{-X}_{\xi\xi}|\dif\xi\\
&\leq C\varepsilon_1\int_\mathbb{R}|(u-\widetilde{u})_{\xi}||v-\widetilde{v}||\widetilde{v}^R_{\xi\xi}|\dif\xi+C\delta_S\int_\mathbb{R}|(u-\widetilde{u})_{\xi}||(\widetilde{v}^S)^{-X}_{\xi}|^{\frac{1}{2}}|v-\widetilde{v}|\dif\xi\\
&\le C\varepsilon_1^2\|\widetilde{v}^R_{\xi\xi}\|_{L^2}+C\delta_S\left(\int_\mathbb{R}(\partial_{\xi}\Psi)^2\dif\xi+\mathcal{G}^S(U)\right).
\end{align*}
So, we derive that
\[
\begin{aligned}
  &\int_\mathbb{R}(\partial_{\xi}\Psi)p^{\prime}(v)(v-\widetilde{v})_{\xi\xi}-(p(v)-p(\widetilde{v}))_{\xi\xi}\dif\xi\\
  &\leq C\delta\int_\mathbb{R}(\partial_{\xi}\Psi)^2\dif\xi
  +C(\varepsilon_1+\delta)\int_\mathbb{R}(\partial_{\xi}\Phi)^2\dif\xi
  +C\delta_S\mathcal{G}^S(U)+C\delta_R^{\frac{3}{2}}\mathcal{G}^R(U)
  +C\varepsilon_1^2\|\widetilde{v}^R_{\xi\xi}\|_{L^2}.
\end{aligned}
\]
From \eqref{79}, using Lemma \ref{Le2}, Lemma \ref{Le4}, we derive that for $k=1,2$,
\begin{equation}\label{82}
\begin{aligned}
&\partial_{\xi}^{k+1}\left(p(\widetilde{v})-p(\widetilde{v}^R)-p((\widetilde{v}^S)^{-X})\right)\\
&\leq C\left(|\widetilde{v}^R_{\xi}||(\widetilde{v}^S)^{-X}-v_m|
  +|(\widetilde{v}^S)^{-X}_{\xi}||\widetilde{v}^R-v_m|
 +|\widetilde{v}^R_{\xi}||(\widetilde{v}^S)^{-X}_{\xi}|\right)
 +C\delta_S\sum\limits_{j=2}\limits^{k+1}|\partial_{\xi}^j\widetilde{v}^R|,
\end{aligned}
\end{equation}
and some tedious calculations give
\begin{align*}
\left(\mu\frac{\widetilde{u}^R_{\xi}}{\widetilde{v}^R}\right)_{\xi\xi}=&\mu\left(\frac{\widetilde{u}^R_{\xi\xi\xi}}{\widetilde{v}^R}
-\frac{2\widetilde{u}^R_{\xi\xi}\widetilde{v}^R_{\xi}}{(\widetilde{v}^R)^2}-\frac{\widetilde{u}^R_{\xi}\widetilde{v}^R_{\xi\xi}}{(\widetilde{v}^R)^2}
+\frac{2\widetilde{u}^R_{\xi}(\widetilde{v}^R_{\xi})^2}{(\widetilde{v}^R)^3}\right),\\
\left(\mu\frac{\widetilde{u}^R_{\xi}}{\widetilde{v}^R}\right)_{\xi\xi\xi}=&\mu\left(\frac{\widetilde{u}^R_{\xi\xi\xi\xi}}{\widetilde{v}^R}
-\frac{3\widetilde{u}^R_{\xi\xi\xi}\widetilde{v}^R_{\xi}+3\widetilde{u}^R_{\xi\xi}\widetilde{v}^R_{\xi\xi}
+\widetilde{u}^R_{\xi}\widetilde{v}^R_{\xi\xi\xi}}{(\widetilde{v}^R)^2}\right.+\frac{6\widetilde{u}^R_{\xi\xi}(\widetilde{v}^R_{\xi})^2+6\widetilde{u}^R_{\xi}\widetilde{v}^R_{\xi\xi}\widetilde{v}^R_{\xi}}{(\widetilde{v}^R)^3}
\left.-\frac{6\widetilde{u}^R_{\xi}(\widetilde{v}^R_{\xi})^3}{(\widetilde{v}^R)^4}\right).
\end{align*}
Then, we have
\[
\begin{aligned}
 & \partial_{\xi}F_1=(p(\widetilde{v})-p(\widetilde{v}^R)-p((\widetilde{v}^S)^{-X}))_{\xi\xi}
 -\left(\mu\frac{\widetilde{u}^R_{\xi}}{\widetilde{v}^R}\right)_{\xi\xi}\\
   &\leq C\left(|\widetilde{v}^R_{\xi}||(\widetilde{v}^S)^{-X}-v_m|
   +|(\widetilde{v}^S)^{-X}_{\xi}||\widetilde{v}^R-v_m|
   +|\widetilde{v}^R_{\xi}||(\widetilde{v}^S)^{-X}_{\xi}|+\sum\limits_{j=2}\limits^{3}|\partial^j_{\xi}\widetilde{u}^R|
   +|\widetilde{u}^R_{\xi}||\widetilde{v}^R_{\xi}|
   +\delta_S|\widetilde{v}^R_{\xi\xi}|\right).
   \end{aligned}
\]
So, using \eqref{45} and H\"{o}lder inequality, we have
\[
\begin{aligned}
  &\int_\mathbb{R}(\partial_{\xi}\Psi)\partial_{\xi}F_1\dif\xi\\
  &\leq C\varepsilon_1\left(\left\||\widetilde{v}^R_{\xi}||(\widetilde{v}^S)^{-X}-v_m|\right\|_{L^2}
   +\left\||(\partial^j_{\xi}\widetilde{v}^S)^{-X}||\widetilde{v}^R-v_m|\right\|_{L^2}\right.
   +\delta_S\|\widetilde{v}^R_{\xi\xi}\|_{L^2}
   +\|\widetilde{u}^R_{\xi\xi}\|_{H^1}\\
    &\left.+\left\||\widetilde{v}^R_{\xi}||(\widetilde{v}^S)^{-X}_{\xi}|\right\|_{L^2}
  +\|\widetilde{u}^R_{\xi}\|_{L^4}\|\widetilde{v}^R_{\xi}\|_{L^4}\right).
  \end{aligned}
\]
Combining the above estimates, we derive that
\[
\begin{aligned}
 &R_7^1\leq
 C(\varepsilon_1+\delta)\int_\mathbb{R}(\partial_{\xi}\Phi)^2\dif\xi
   +C\delta_S\mathcal{G}^S(U)+C\delta_R^{\frac{1}{2}}\mathcal{G}^R(U)
   +C\delta\int_\mathbb{R}(\partial_{\xi}\Psi)^2\dif\xi\\
& +C\varepsilon_1\left(\left\||\widetilde{v}^R_{\xi}||(\widetilde{v}^S)^{-X}-v_m|\right\|_{L^2}
   +\left\||(\partial^j_{\xi}\widetilde{v}^S)^{-X}||\widetilde{v}^R-v_m|\right\|_{L^2}\right)
   +C\|\widetilde{v}^R_{\xi\xi}\|_{L^2}\\
   &+C\varepsilon_1\left(\|\widetilde{u}^R_{\xi\xi}\|_{H^1}
   +\left\||\widetilde{v}^R_{\xi}||(\widetilde{v}^S)^{-X}_{\xi}|\right\|_{L^2}
   +\|\widetilde{u}^R_{\xi}\|_{L^4}\|\widetilde{v}^R_{\xi}\|_{L^4}\right).
\end{aligned}
\]
Similarly, for $k=2$, since
\[
\begin{aligned}
  &p^{\prime}(v)(v-\widetilde{v})_{\xi\xi}-(p(v)-p(\widetilde{v}))_{\xi\xi}\\
   &=-p^{\prime}(v)\widetilde{v}_{\xi\xi\xi}-p^{\prime\prime\prime}(v)v^3_{\xi}-3p^{\prime\prime}(v)v_{\xi}v_{\xi\xi}
                    +p^{\prime\prime\prime}(\widetilde{v})\widetilde{v}^3_{\xi}
                    +3p^{\prime\prime}(\widetilde{v})\widetilde{v}_{\xi\xi}\widetilde{v}_{\xi}
                    +p^{\prime}(\widetilde{v})\widetilde{v}_{\xi\xi\xi}\\
   &=-p^{\prime\prime\prime}(v)(v_{\xi}-\widetilde{v}_{\xi})^3
   -(p^{\prime\prime\prime}(v)-p^{\prime\prime\prime}(\widetilde{v}))\widetilde{v}_{\xi}^3
   -3p^{\prime\prime\prime}(v)\widetilde{v}_{\xi}^2(v_{\xi}-\widetilde{v}_{\xi})
   -3p^{\prime\prime\prime}(v)\widetilde{v}_{\xi}(v_{\xi}-\widetilde{v}_{\xi})^2\\
   &\quad-3p^{\prime\prime}(v)(v_{\xi}-\widetilde{v}_{\xi})(v_{\xi\xi}
   -\widetilde{v}_{\xi\xi})-3p^{\prime\prime}(v)\widetilde{v}_{\xi}(v_{\xi\xi}-\widetilde{v}_{\xi\xi})
  -3p^{\prime\prime}(v)\widetilde{v}_{\xi\xi}(v_{\xi}-\widetilde{v}_{\xi})\\
   &\quad-3(p^{\prime\prime}(v)-p^{\prime\prime}(\widetilde{v}))\widetilde{v}_{\xi}\widetilde{v}_{\xi\xi}-(p^{\prime}(v)-p^{\prime}(\widetilde{v}))\widetilde{v}_{\xi\xi\xi},
\end{aligned}
\]
 we have
\[
\begin{aligned}
&\int_\mathbb{R}(\partial_{\xi\xi}\Psi)(p^{\prime}(v)(v-\widetilde{v})_{\xi\xi}-(p(v)-p(\widetilde{v}))_{\xi\xi})\dif\xi\\
&\leq \delta\mathcal{G}^S(U)
+C(\varepsilon_1+\delta)\left(\|\partial_{\xi}\Phi\|^2_{H^1}
+\int_\mathbb{R}(\partial_{\xi\xi}\Psi)^2\dif\xi\right)
+C\varepsilon_1^2\|\widetilde{v}^R_{\xi\xi}\|_{H^1}.
\end{aligned}
\]
Using \eqref{82}, we have
\[
\begin{aligned}
  &\partial_{\xi\xi}F_1=(p(\widetilde{v})-p(\widetilde{v}^R)
  -p((\widetilde{v}^S)^{-X}))_{\xi\xi\xi}-\left(\mu\frac{\widetilde{u}^R_{\xi}}{\widetilde{v}^R}\right)_{\xi\xi\xi}\leq
  C\left(|\widetilde{v}^R_{\xi\xi\xi}|+|\widetilde{v}^R_{\xi\xi}|
  +|\widetilde{v}^R_{\xi\xi}||\widetilde{u}^R_{\xi\xi}|\right)\\
   & +C\left(|\widetilde{v}^R_{\xi}||(\widetilde{v}^S)^{-X}-v_m|
   +|(\partial^j_{\xi}\widetilde{v}^S)^{-X}||\widetilde{v}^R-v_m|
   +\sum\limits_{j=2}\limits^{4}|\partial^j_{\xi}\widetilde{u}^R|+|\widetilde{v}^R_{\xi}||(\widetilde{v}^S)^{-X}_{\xi}|
   +|\widetilde{u}^R_{\xi}||\widetilde{v}^R_{\xi}|\right).
\end{aligned}
\]
Similarly, we have
\[
\begin{aligned}
  \int_\mathbb{R}(\partial_{\xi\xi}\Psi)\partial_{\xi\xi}F_1\dif\xi
\leq& C\varepsilon_1\left(\left\|\widetilde{v}^R_{\xi}|(\widetilde{v}^S)^{-X}-v_m|\right\|_{L^2}
   +\left\||(\widetilde{v}^S)^{-X}_{\xi}||\widetilde{v}^R-v_m|\right\|_{L^2}
   +\left\||\widetilde{v}^R_{\xi}||(\widetilde{v}^S)^{-X}_{\xi}|\right\|_{L^2}\right)\\
&   +C\left(\|\widetilde{u}^R_{\xi\xi}\|_{H^2}+\|\widetilde{v}^R_{\xi\xi}\|_{H^1}
+\|\widetilde{u}^R_{\xi}\|_{L^4}\|\widetilde{v}^R_{\xi}\|_{L^4}
+\|\widetilde{u}^R_{\xi\xi}\|_{L^4}\|\widetilde{v}^R_{\xi\xi}\|_{L^4}\right).
\end{aligned}
\]
Thus, we get
\[
\begin{aligned}
&R^2_7\leq
C(\varepsilon_1+\delta)\left(\|\partial_{\xi}\Phi\|^2_{H^1}
+\int_\mathbb{R}(\partial_{\xi\xi}\Psi)^2\dif\xi\right)
+C\varepsilon_1\left(\left\|\widetilde{v}^R_{\xi}|(\widetilde{v}^S)^{-X}-v_m|\right\|_{L^2}
   +\left\||(\widetilde{v}^S)^{-X}_{\xi}||\widetilde{v}^R-v_m|\right\|_{L^2}\right.\\
   &+\|\widetilde{u}^R_{\xi}\|_{L^4}\|\widetilde{v}^R_{\xi}\|_{L^4}
   +\|\widetilde{u}^R_{\xi\xi}\|_{H^2}+\|\widetilde{v}^R_{\xi\xi}\|_{H^1}
   +\|\widetilde{u}^R_{\xi\xi}\|_{L^4}\|\widetilde{v}^R_{\xi\xi}\|_{L^4}
   +\left.\left\||\widetilde{v}^R_{\xi}||(\widetilde{v}^S)^{-X}_{\xi}|\right\|_{L^2}\right)
   +\delta\mathcal{G}^S(U).
\end{aligned}
\]
For $R_8^1$, since
\[
v(\Pi-\widetilde{\Pi})_{\xi}-(v\Pi-\widetilde{v}\widetilde{\Pi})_{\xi}=-(v-\widetilde{v})\widetilde{\Pi}_{\xi}
-(v-\widetilde{v})_{\xi}(\Pi-\widetilde{\Pi})
-(v-\widetilde{v})_{\xi}\widetilde{\Pi}-\widetilde{v}_{\xi}(\Pi-\widetilde{\Pi}),
\]
we have
\[
\begin{aligned}
&\int_\mathbb{R}\frac{1}{\mu}(v(\Pi-\widetilde{\Pi})_{\xi}-(v\Pi-\widetilde{v}\widetilde{\Pi})_{\xi})(\Pi_{\xi}
-\widetilde{\Pi}_{\xi})\dif\xi\\
&=-\int_\mathbb{R}\frac{1}{\mu}(v-\widetilde{v})\widetilde{\Pi}_{\xi}(\Pi_{\xi}-\widetilde{\Pi}_{\xi})\dif\xi
-\int_\mathbb{R}\frac{1}{\mu}(v-\widetilde{v})_{\xi}(\Pi-\widetilde{\Pi})(\Pi_{\xi}-\widetilde{\Pi}_{\xi})\dif\xi\\
&\quad-\int_\mathbb{R}\frac{1}{\mu}(v-\widetilde{v})_{\xi}\widetilde{\Pi}(\Pi_{\xi}-\widetilde{\Pi}_{\xi})\dif\xi
-\int_\mathbb{R}\frac{1}{\mu}\widetilde{v}_{\xi}(\Pi-\widetilde{\Pi})(\Pi_{\xi}-\widetilde{\Pi}_{\xi})\dif\xi\\
&=:L_1+L_2+L_3+L_4.
\end{aligned}
\]
Using \eqref{45}, H\"{o}lder inequality and Young's inequality, we have
\[
\begin{aligned}
 L_1&\leq C\int_\mathbb{R}|v-\widetilde{v}||(\widetilde{\Pi}^S)_{\xi}^{-X}||\Pi_{\xi}-\widetilde{\Pi}_{\xi}|\dif\xi
 +C\int_\mathbb{R}|v-\widetilde{v}|\left|\left(\frac{\widetilde{u}^R_{\xi}}{\widetilde{v}^R}\right)_{\xi\xi}\right||\Pi_{\xi}
 -\widetilde{\Pi}_{\xi}|\dif\xi\\
&\leq C\int_\mathbb{R}|v-\widetilde{v}||(\widetilde{v}^S)_{\xi}^{-X}||\Pi_{\xi}-\widetilde{\Pi}_{\xi}|\dif\xi
 +C\varepsilon_1\int_\mathbb{R}(|\widetilde{u}^R_{\xi\xi}|
 +|\widetilde{u}^R_{\xi}||\widetilde{v}^R_{\xi}|)|\Pi_{\xi}-\widetilde{\Pi}_{\xi}|\dif\xi\\
&\leq C(\delta_S+\varepsilon_1)\int_\mathbb{R}\frac{v}{\mu}(\partial_{\xi}Q)^2\dif\xi+C\delta_S\mathcal{G}^S(U)
  +C\varepsilon_1\left(\|\widetilde{u}^R_{\xi\xi}\|^2_{L^2(\mathbb{R})}
  +\|\widetilde{u}^R_{\xi}\|^2_{L^4}\|\widetilde{v}^R_{\xi}\|^2_{L^4}\right),
\end{aligned}
\]
\[
\begin{aligned}
   L_2&\leq C\varepsilon_1\left(\int_\mathbb{R}\frac{v}{\mu}|\Pi-\widetilde{\Pi}|^2\dif\xi
   +\int_\mathbb{R}\frac{v}{\mu}|\Pi_{\xi}-\widetilde{\Pi}_{\xi}|^2\dif\xi\right)
   =C\varepsilon_1\left(G(U)
   +\int_\mathbb{R}\frac{v}{\mu}|\partial_{\xi}Q|^2\dif\xi\right),
\end{aligned}
\]
\[
L_3\leq C\delta\left(\int_\mathbb{R}|\partial_{\xi}\Phi|^2\dif\xi
   +\int_\mathbb{R}\frac{v}{\mu}|\partial_{\xi}Q|^2\dif\xi\right),
\quad
L_4\leq C\delta\left(G(U)+\int_\mathbb{R}\frac{v}{\mu}|\partial_{\xi}Q|^2\dif\xi\right).
\]
Therefore, we get
\[
\begin{aligned}
&\int_\mathbb{R}\frac{1}{\mu}(v(\Pi-\widetilde{\Pi})_{\xi}-(v\Pi-\widetilde{v}\widetilde{\Pi})_{\xi})Q_{\xi}\dif\xi
 \leq C(\delta+\varepsilon_1)\left(G(U)+\int_\mathbb{R}\frac{v}{\mu}|\partial_{\xi}Q|^2\dif\xi\right)\\
 &+C\delta_S\mathcal{G}^S(U)+C\delta\int_\mathbb{R}|\partial_{\xi}\Phi|^2\dif\xi
 +C\varepsilon_1\left(\|\widetilde{u}^R_{\xi\xi}\|^2_{L^2}
  +\|\widetilde{u}^R_{\xi}\|^2_{L^4}\|\widetilde{v}^R_{\xi}\|^2_{L^4}\right).
\end{aligned}
\]
On the other hand, we have
\[
\begin{aligned}
&\left(\mu\frac{\widetilde{u}^R_{\xi}}{\widetilde{v}^R}\right)_{t\xi}=\mu\frac{(\widetilde{u}^R_{\xi\xi t}\widetilde{v}^R
+\widetilde{u}^R_{\xi t}\widetilde{v}^R_{\xi}-\widetilde{u}^R_{\xi\xi}\widetilde{v}^R_{t}
-\widetilde{u}^R_{\xi}\widetilde{v}^R_{t\xi})(\widetilde{v}^R)^2
-(\widetilde{u}^R_{\xi t}\widetilde{v}^R-\widetilde{u}^R_{\xi}\widetilde{v}^R_{t})2\widetilde{v}^R\widetilde{v}^R_{\xi}
}{(\widetilde{v}^R)^4},\\
&\left((\widetilde{v}^R-v_m)(\widetilde{\Pi}^S)^{-X}\right)_{\xi}=\widetilde{v}^R_{\xi}(\widetilde{\Pi}^S)^{-X}
+(\widetilde{v}^R-v_m)(\widetilde{\Pi}^S)^{-X}_{\xi},\\
&\left(((\widetilde{v}^S)^{-X}-v_m)\left(\mu\frac{\widetilde{u}^R_{\xi}}{\widetilde{v}^R}\right)\right)_{\xi}
=((\widetilde{v}^S)^{-X}-v_m)\left(\mu\frac{\widetilde{u}^R_{\xi}}{\widetilde{v}^R}\right)_{\xi}
+(\widetilde{v}^S)^{-X}_{\xi}\left(\mu\frac{\widetilde{u}^R_{\xi}}{\widetilde{v}^R}\right).
\end{aligned}
\]
Using Lemma \ref{Le2} and Lemma \ref{Le4}, we derive that
\[
\begin{aligned}
\partial_{\xi}F_2&\leq C\tau\left(|\widetilde{u}^R_{\xi\xi\xi}|+|\widetilde{u}^R_{\xi\xi}|+|\widetilde{v}^R_{\xi\xi\xi}|+|\widetilde{v}^R_{\xi\xi}|
+|\widetilde{u}^R_{\xi}|^2+|\widetilde{u}^R_{\xi}||\widetilde{v}^R_{\xi}|\right)
+C\delta_S\left(|\widetilde{u}^R_{\xi}||\widetilde{v}^R_{\xi}|+|\widetilde{u}^R_{\xi\xi}|\right)\\
&+C\left(|(\widetilde{v}^S)^{-X}_{\xi}||\widetilde{v}^R-v_m|
+|(\widetilde{v}^S)^{-X}_{\xi}||\widetilde{v}^R_{\xi}|\right)
.
\end{aligned}
\]
Using Young's inequality, we have
\[
\begin{aligned}
&\int_\mathbb{R}\frac{1}{\mu}\partial_{\xi}F_2Q_{\xi}\dif\xi\leq \frac{1}{64}\int_\mathbb{R}\frac{v}{\mu}|\partial_{\xi}Q|^2\dif\xi
+C\left(\|(\widetilde{u}^R_{\xi\xi},\widetilde{v}^R_{\xi\xi})\|^2_{H^1}
+\|\widetilde{u}^R_{\xi}\|^4_{L^4}
+\|\widetilde{u}^R_{\xi}\|^2_{L^4}\|\widetilde{v}^R_{\xi}\|^2_{L^4}\right)\\
&+C\left(\left\||(\widetilde{v}^S)^{-X}_{\xi}||\widetilde{v}^R-v_m|\right\|^2_{L^2(\mathbb{R})}
+\left\||(\widetilde{v}^S)^{-X}_{\xi}||\widetilde{v}^R_{\xi}|\right\|^2_{L^2(\mathbb{R})}\right).
\end{aligned}
\]
Thus, combining the above estimates, we get
\[
\begin{aligned}
&R_8^1\leq
C(\delta+\varepsilon_1)G(U)+\frac{1}{32}\int_\mathbb{R}\frac{v}{\mu}|\partial_{\xi}Q|^2\dif\xi
+C\delta_S\mathcal{G}^S(U)+C\delta\int_\mathbb{R}|\partial_{\xi}\Phi|^2\dif\xi\\
 &+C\left(\|(\widetilde{u}^R_{\xi\xi},\widetilde{v}^R_{\xi\xi})\|^2_{H^1}
+\|\widetilde{u}^R_{\xi}\|^4_{L^4}
+\|\widetilde{u}^R_{\xi}\|^2_{L^4}\|\widetilde{v}^R_{\xi}\|^2_{L^4}+\left\||(\widetilde{v}^S)^{-X}_{\xi}||\widetilde{v}^R-v_m|\right\|^2_{L^2}
+\left\||(\widetilde{v}^S)^{-X}_{\xi}||\widetilde{v}^R_{\xi}|\right\|^2_{L^2}\right).
\end{aligned}
\]
Next, we estimate $R_8^2$.  Firstly, we note that
\[
\begin{aligned}
v(\Pi-\widetilde{\Pi})_{\xi}-(v\Pi-\widetilde{v}\widetilde{\Pi})_{\xi}=&-(v-\widetilde{v})_{\xi\xi}(\Pi-\widetilde{\Pi})
-(v-\widetilde{v})_{\xi\xi}\widetilde{\Pi}-\widetilde{v}_{\xi\xi}(\Pi-\widetilde{\Pi})
-(v-\widetilde{v})\widetilde{\Pi}_{\xi\xi}\\
&-2(v-\widetilde{v})_{\xi}(\Pi-\widetilde{\Pi})_{\xi}-2(v-\widetilde{v})_{\xi}\widetilde{\Pi}_{\xi}
-2\widetilde{v}_{\xi}(\Pi-\widetilde{\Pi})_{\xi}.
\end{aligned}
\]
Using H\"{o}lder inequality, Sobolev's imbedding theorem and Young's inequality, we have
\[
\begin{aligned}
&\int_\mathbb{R}(v-\widetilde{v})_{\xi\xi}(\Pi-\widetilde{\Pi})(\Pi-\widetilde{\Pi})_{\xi\xi}\dif\xi\\
&\leq C
\|(v-\widetilde{v})_{\xi\xi}\|_{L^2}
\left(\int_\mathbb{R}\left(\frac{v}{\mu}\right)^2(\Pi-\widetilde{\Pi})^2\left((\Pi-\widetilde{\Pi})_{\xi\xi}\right)^2\dif\xi\right)^{\frac{1}{2}}\\
&\leq C\varepsilon_1|\sqrt{\frac{v}{\mu}}(\Pi-\widetilde{\Pi})|_{L^{\infty}}
\|\sqrt{\frac{v}{\mu}}(\Pi-\widetilde{\Pi})_{\xi\xi}\|_{L^2}\\
&\leq C\varepsilon_1\left(\|\sqrt{\frac{v}{\mu}}(\Pi-\widetilde{\Pi})\|_{H^1}^2+
\|\sqrt{\frac{v}{\mu}}(\Pi-\widetilde{\Pi})_{\xi\xi}\|_{L^2}^2\right),
\end{aligned}
\]
On the other hand, by the definition of $F_2$, we have
\[
\begin{aligned}
&\partial_{\xi\xi}F_2\leq
C\left(|(\widetilde{v}^S)^{-X}_{\xi}||\widetilde{v}^R-v_m|
+|(\widetilde{v}^S)^{-X}_{\xi}||\widetilde{v}^R_{\xi}|\right)
+\delta_S\left(|\widetilde{u}^R_{\xi\xi\xi}|+
|\widetilde{u}^R_{\xi\xi}|+|\widetilde{u}^R_{\xi}||\widetilde{v}^R_{\xi}|
+|\widetilde{v}^R_{\xi\xi}|\right)\\
&C\tau\left(|\widetilde{u}^R_{\xi\xi\xi\xi}|+
|\widetilde{u}^R_{\xi\xi\xi}|+|\widetilde{u}^R_{\xi\xi}|+|\widetilde{v}^R_{\xi\xi\xi\xi}|
+|\widetilde{v}^R_{\xi\xi\xi}|+|\widetilde{v}^R_{\xi\xi}|
+|\widetilde{u}^R_{\xi\xi}|^2+|\widetilde{v}^R_{\xi\xi}|^2
+|\widetilde{u}^R_{\xi\xi}||\widetilde{v}^R_{\xi\xi}|
+|\widetilde{u}^R_{\xi}|^2+|\widetilde{u}^R_{\xi}||\widetilde{v}^R_{\xi}|\right).
\end{aligned}
\]
Combining the above estimates, and using similar methods as in the estimates of $R_8^1$, we derive that
\[
\begin{aligned}
&R_8^2\leq
C(\delta+\varepsilon_1)\left(G(U)+\int_\mathbb{R}\frac{v}{\mu}|\partial_{\xi}Q|^2\dif\xi\right)
+\frac{1}{32}\int_\mathbb{R}\frac{v}{\mu}|\partial_{\xi\xi}Q|^2\dif\xi
 +C\delta_S\mathcal{G}^S(U)+C\delta\|\partial_{\xi}\Phi\|_{H^1}^2\\
&+C\left(\|(\widetilde{u}^R_{\xi\xi},\widetilde{v}^R_{\xi\xi})\|^2_{H^2}
+\|\widetilde{u}^R_{\xi\xi}\|^4_{L^4}+\|\widetilde{v}^R_{\xi\xi}\|^4_{L^4}
+\|\widetilde{u}^R_{\xi}\|^4_{L^4}
+\|\widetilde{u}^R_{\xi}\|^2_{L^4}\|\widetilde{v}^R_{\xi}\|^2_{L^4}
+\|\widetilde{u}^R_{\xi\xi}\|^2_{L^4}\|\widetilde{v}^R_{\xi\xi}\|^2_{L^4}\right)\\
&+C\left(\left\||(\widetilde{v}^S)^{-X}_{\xi}||\widetilde{v}^R-v_m|\right\|^2_{L^2}
+\left\||(\widetilde{v}^S)^{-X}_{\xi}||\widetilde{v}^R_{\xi}|\right\|^2_{L^2}\right).
\end{aligned}
\]
Finally, integrating the equality \eqref{new1}  over $[0,t]$,  using the above estimates  and recalling Lemma \ref{Le4}, Lemma \ref{Le9}, we complete the proof of this lemma.
\end{proof}

\subsection{Dissipative estimates}

In the following lemmas, we give the dissipative estimates of given solutions to system \eqref{50}.
\begin{lemma}\label{Le11}
There exists a constant $C,C_1,\epsilon>0$ such that for $0\leq t\leq T$, we have
\[
\begin{aligned}
&\int_0^t\|\partial_{\xi}\Phi\|_{H^1}^2\dif t\leq\epsilon\|\Psi\|_{H^1}^2
+C(\epsilon)\|\partial_{\xi}\Phi\|_{H^1}^2
     +C\delta_S^2\int_0^t|\dot{X}(t)|^2\dif t+C\int_0^t\|\partial_{\xi}Q\|_{H^1}^2\dif t\\
     &+C\left(\|\Psi_0\|_{H^1}^2+\|\partial_{\xi}\Phi_0\|_{H^1}^2\right)
     +C_1\int_0^t\|\partial_{\xi}\Psi\|_{H^1}^2\dif t
     +C\int_0^t\left(\delta_S \mathcal{G}^S(U)
     +\delta_R^{\frac{1}{2}}\mathcal{G}^R(U)\right)\dif t+C\delta_R^{\theta},
\end{aligned}
\]
where $\theta=\min\{\frac{1}{2}, \frac{3}{2}-\frac{1}{q}\}$.
\end{lemma}
\begin{proof}
Multiplying the equation $\eqref{80}_2$ by $\partial_{\xi}^{k+1}\Phi$ for $k=0 , 1,$ and integrating over $(0,t)\times \mathbb{R}$, we get
\[
\int_0^t\int_{\mathbb{R}}-p^{\prime}(v)\left(\partial_{\xi}^{k+1}\Phi\right)^2\dif\xi \dif t=:\sum\limits_{i=0}\limits^{5}M_i^k,
\]
where
\[
M_1^k=\int_0^t\int_\mathbb{R}\partial_{t}\partial_{\xi}^k\Psi\partial_{\xi}^{k+1}\Phi \dif\xi \dif t,
\quad M_2^k=-\int_0^t\int_\mathbb{R}\sigma\partial_{\xi}^{k+1}\Psi\partial_{\xi}^{k+1}\Phi \dif\xi \dif t,
\]
\[
M_3^k=-\int_0^t\int_\mathbb{R}\dot{X}(t)\partial_{\xi}^{k+1}(\widetilde{u}^S)^{-X}\partial_{\xi}^{k+1}\Phi \dif\xi \dif t,
\quad M_4^k=-\int_0^t\int_\mathbb{R}\partial_{\xi}^{k+1}Q\partial_{\xi}^{k+1}\Phi \dif\xi \dif t,
\]
\[
M_5^k=-\int_0^t\int_\mathbb{R}F_3^k\partial_{\xi}^{k+1}\Phi \dif\xi \dif t.
\]
Firstly, doing integration by part and using equation $\eqref{80}_1$, we get
\[
\begin{aligned}
&M_1^k+M_2^k\\
&=\int_\mathbb{R}\left(\partial_{\xi}^k\Psi(t)\partial_{\xi}^{k+1}\Phi(t) -\partial_{\xi}^k\Psi_0\partial_{\xi}^{k+1}\Phi_0 \right)\dif\xi-\int_0^t\int_\mathbb{R}\partial_{\xi}^k\Psi
      \left(\dot{X}(t)\partial_{\xi}^{k+2}(\widetilde{v}^S)^{-X}
      +\partial_{\xi}^{k+2}\Psi\right)\dif\xi \dif t.
\end{aligned}
\]
Since
\[
\begin{aligned}
 -\int_0^t\int_\mathbb{R}\partial_{\xi}^k\Psi
      \dot{X}(t)\partial_{\xi}^{k+2}(\widetilde{v}^S)^{-X}\dif\xi \dif t&=\int_0^t\int_\mathbb{R}\partial_{\xi}^{k+1}\Psi
      \dot{X}(t)\partial_{\xi}^{k+1}(\widetilde{v}^S)^{-X}\dif\xi \dif t\\
   &\leq C\delta_S^2\int_0^t|\dot{X}(t)|^2\dif t+C\delta_S\int_0^t\int_\mathbb{R}\left(\partial_{\xi}^{k+1}\Psi\right)^2\dif\xi \dif t,
\end{aligned}
\]
and
\[
-\int_0^t\int_\mathbb{R}\partial_{\xi}^k\Psi
      \partial_{\xi}^{k+2}\Psi \dif\xi\dif t
      =\int_0^t\int_\mathbb{R}\left(\partial_{\xi}^{k+1}\Psi\right)^2\dif\xi \dif t,
\]
which gives
\[
\begin{aligned}
M_1^k+M_2^k\leq&\epsilon\int_\mathbb{R}\left(\partial_{\xi}^k\Psi(t)\right)^2\dif\xi
+C(\epsilon)\int_\mathbb{R}\left(\partial_{\xi}^{k+1}\Phi(t)\right)^2\dif\xi
+C\int_0^t\int_\mathbb{R}\left(\partial_{\xi}^{k+1}\Psi\right)^2\dif\xi \dif t\\
&+C\left(\delta_S^2\int_0^t|\dot{X}(t)|^2\dif t+\int_\mathbb{R}\left(\partial_{\xi}^k\Psi_0\right)^2\dif\xi
+\int_\mathbb{R}\left(\partial_{\xi}^{k+1}\Phi_0\right)^2\dif\xi\right).
\end{aligned}
\]
Secondly, for $M_3^k$ and $M_4^k$, using Young's inequality, we have
\[
M_3^k\leq C\delta_S^2\int_0^t|\dot{X}(t)|^2\dif t+C\delta_S\int_0^t\int_\mathbb{R}\left(\partial_{\xi}^{k+1}\Phi\right)^2\dif\xi \dif t,
\]
and
\[
M_4^k\leq \frac{1}{2}\int_0^t\int_{\mathbb{R}}-p^{\prime}(v)\left(\partial_{\xi}^{k+1}\Phi\right)^2\dif\xi \dif t
   +C\int_0^t\int_\mathbb{R}\left(\partial_{\xi}^{k+1}Q\right)^2\dif\xi \dif t.
\]
Next, we estimate $M_5^k$. For $k=0$,  we have
\[
F_3^0=-(p^{\prime}(v)-p^{\prime}(\widetilde{v}))\widetilde{v}_{\xi}-\left(p(\widetilde{v})-p(\widetilde{v}^R)-p((\widetilde{v}^S)^{-X})\right)_{\xi}
-\mu\left(\frac{\widetilde{u}^R_{\xi}}{\widetilde{v}^R}\right)_{\xi},
\]
which gives
\[
M_5^0\leq C\delta^{\frac{1}{2}}\int_0^t\int_{\mathbb{R}}\left(\partial_{\xi}\Phi\right)^2\dif\xi \dif t
+C\delta_S\int_0^t\mathcal{G}^S(U)\dif t+C\delta_R^{\frac{1}{2}}\int_0^t\mathcal{G}^R(U)\dif t
+C\delta_R^{\theta}.
\]
For $k=1$, the term $F_3^1$ is the same as in Lemma \ref{Le10}.  Therefore, using similar estimates as in Lemma \ref{Le10}, we have
\[
\begin{aligned}
M_5^1\leq &C(\delta+\varepsilon_1)\int_0^t\|\partial_\xi \Phi\|_{H^1}^2 \dif t+C\delta_S\int_0^t\mathcal{G}^S(U)\dif t+C\delta_R^{\frac{1}{2}}\int_0^t\mathcal{G}^R(U)\dif t
   +C\delta_R^{\theta}.
\end{aligned}
\]
Therefore, combining the above estimates, we get the desired results.
\end{proof}
\begin{lemma}\label{Le12}
There exists a constant $C,\epsilon_1>0$ such that for $0\leq t\leq T$, we have
\[
\begin{aligned}
&\int_0^t\|\partial_{\xi}\Psi\|_{H^1}^2\dif t\leq\epsilon_1\tau\|Q\|_{H^1}^2
+C(\epsilon_1)\|\partial_{\xi}\Psi\|_{H^1}^2
     +C\delta_S^2\int_0^t|\dot{X}(t)|^2\dif t
     +C\left(\|Q_0\|_{H^1}^2+\|\partial_{\xi}\Psi_0\|_{H^1}^2\right)\\
     &+\frac{1}{2C_1}\int_0^t\|\partial_{\xi}\Phi\|_{H^1}^2\dif t
     +C\int_0^t\|\partial_{\xi}Q\|_{H^1}^2\dif t+C\delta_S\int_0^t\mathcal{G}^S(U)\dif t+C\delta_R^{\frac{1}{2}}\int_0^t\mathcal{G}^R(U)\dif t
     +C\delta_R^{\theta},
\end{aligned}
\]
where $C_1$ is the constant in Lemma \ref{Le11} and $\theta=\min\{\frac{1}{2}, \frac{3}{2}-\frac{1}{q}\}$.
\end{lemma}
\begin{proof}
Multiplying the equation $\eqref{80}_3$ by $\partial_{\xi}^{k+1}\Psi$ for $k=0 , 1$, and integrating over $(0,t)\times\mathbb{R}$, we get
\[
\int_0^t\int_{\mathbb{R}}\mu\left(\partial_{\xi}^{k+1}\Psi\right)^2\dif\xi \dif t=:\sum\limits_{i=0}\limits^{5}N_i^k,
\]
where
\[
N_1^k=\int_0^t\int_\mathbb{R}\tau\partial_{t}\partial_{\xi}^kQ\partial_{\xi}^{k+1}\Psi \dif\xi \dif t,
\quad N_2^k=-\int_0^t\int_\mathbb{R}\sigma\tau\partial_{\xi}^{k+1}Q\partial_{\xi}^{k+1}\Psi \dif\xi \dif t,
\]
\[
N_3^k=-\int_0^t\int_\mathbb{R}\tau\dot{X}(t)\partial_{\xi}^{k+1}(\widetilde{\Pi}^S)^{-X}\partial_{\xi}^{k+1}\Psi \dif\xi \dif t,
\quad N_4^k=\int_0^t\int_\mathbb{R}v\partial_{\xi}^{k}Q\partial_{\xi}^{k+1}\Psi \dif\xi \dif t,
\]
\[
N_5^k=-\int_0^t\int_\mathbb{R}F_4^k\partial_{\xi}^{k+1}\Psi \dif\xi \dif t.
\]
Firstly, doing integration by part and using equation $\eqref{50}_2$, we get
\[
\begin{aligned}
N_1^k+N_2^k=&\tau\int_\mathbb{R}\partial_{\xi}^kQ(t)\partial_{\xi}^{k+1}\Psi(t) \dif\xi-\tau\int_\mathbb{R}\partial_{\xi}^kQ_0\partial_{\xi}^{k+1}\Psi_0 \dif\xi\\
      &-\tau\int_0^t\int_\mathbb{R}\partial_{\xi}^kQ
      \left(\dot{X}(t)\partial_{\xi}^{k+2}(\widetilde{u}^S)^{-X}
          -\partial_{\xi}^{k+2}(p(v)-p(\widetilde{v}))+\partial_{\xi}^{k+2}Q+\partial^{k+2}_{\xi}F_1\right)\dif\xi \dif t,
\end{aligned}
\]
Using Young's inequality, we get
\[
\begin{aligned}
& -\tau\int_0^t\int_\mathbb{R}\partial_{\xi}^kQ
      \dot{X}(t)\partial_{\xi}^{k+2}(\widetilde{u}^S)^{-X}\dif\xi \dif t\\
      &=\tau\int_0^t\int_\mathbb{R}\partial_{\xi}^{k+1}Q
      \dot{X}(t)\partial_{\xi}^{k+1}(\widetilde{u}^S)^{-X}\dif\xi \dif t\leq C\delta_S^2\tau^2\int_0^t|\dot{X}|^2dt
   +C\delta_S\int_0^t\int_\mathbb{R}\left(\partial_{\xi}^{k+1}Q\right)^2\dif\xi \dif t.
\end{aligned}
\]
Doing integration by part, using Young's inequality, we get
\[
\begin{aligned}
&\tau\int_0^t\int_\mathbb{R}\partial_{\xi}^kQ\partial_{\xi}^{k+2}(p(v)-p(\widetilde{v}))\dif\xi \dif t\\
&=-\tau\int_0^t\int_\mathbb{R}\partial_{\xi}^{k+1}Q\partial_{\xi}^{k+1}(p(v)-p(\widetilde{v}))\dif\xi \dif t\\
&\leq \frac{\mu}{4C_2}\int_0^t\int_\mathbb{R}\left(\partial_{\xi}^{k+1}\Phi\right)^2\dif\xi \dif t
+C\tau^2\int_0^t\int_\mathbb{R}\left(\partial_{\xi}^{k+1}Q\right)^2\dif\xi \dif t
+C\delta_S\int_0^t\mathcal{G}^S(U)\dif t+C\delta_R^{\frac{1}{2}}\int_0^t\mathcal{G}^R(U)\dif t.
\end{aligned}
\]
Recalling the estimates of $\partial_\xi^{k+1} F_1, k=0, 1$ in Lemma \ref{Le10}, we get
\begin{align*}
-\tau\int_0^t\int_\mathbb{R}\partial_{\xi}^kQ
      \partial_{\xi}^{k+2}F_1 \dif\xi \dif t&=\tau\int_0^t\int_\mathbb{R}\partial_{\xi}^{k+1}Q
      \partial_{\xi}^{k+1}F_1 \dif\xi \dif t\\
     & \leq C\int_0^t\int_\mathbb{R}\left(\partial_{\xi}^{k+1}Q\right)^2
       \dif\xi \dif t+C\tau^2\int_0^t\|\partial_{\xi}^{k+1}F_1\|_{L^2}^2\dif t\\
      &\leq C\int_0^t\int_\mathbb{R}\left(\partial_{\xi}^{k+1}Q\right)^2
       \dif\xi \dif t+C\delta_R^{\theta},
\end{align*}
Therefore, we conclude that
\[
\begin{aligned}
&N_1^k+N_2^k\leq\epsilon_1\tau\int_\mathbb{R}\left(\partial_{\xi}^kQ(t)\right)\dif \xi
+C(\epsilon_1)\int_\mathbb{R}\left(\partial_{\xi}^{k+1}\Psi(t)\right)\dif \xi
+C\tau^2\int_0^t\int_\mathbb{R}\left(\partial_{\xi}^{k+1}Q\right)^2\dif\xi \dif t\\
&+C\left(\delta_S^2\int_0^t|\dot{X}(t)|^2\dif t+\int_\mathbb{R}\left(\partial_{\xi}^kQ_0\right)\dif \xi
+\int_\mathbb{R}\left(\partial_{\xi}^{k+1}\Psi_0\right)\dif \xi\right)
+C\delta_S\int_0^t\mathcal{G}^S(U)\dif t\\
&+C\delta_R^{\frac{1}{2}}\int_0^t\mathcal{G}^R(U)\dif t
+\frac{\mu}{4C_2}\int_0^t\int_\mathbb{R}\left(\partial_{\xi}^{k+1}\Phi\right)^2\dif\xi \dif t+C\delta_R^{\theta}.
\end{aligned}
\]
Secondly, for $N_3^k$, we have
\[
N_3^k\leq C\delta_S^2\tau^2\int_0^t|\dot{X}|^2\dif t
+C\delta_S\int_0^t\int_\mathbb{R}\left(\partial_{\xi}^{k+1}\Psi\right)^2\dif\xi \dif t.
\]
For $N_4^k$, using Young's inequality, we have
\[
N_4^k\leq \frac{1}{2}\int_0^t\int_{\mathbb{R}}\mu\left(\partial_{\xi}^{k+1}\Psi\right)^2\dif\xi \dif t
   +C\int_0^t\int_\mathbb{R}\left(\partial_{\xi}^{k}Q\right)^2\dif\xi \dif t.
\]
$N_5^k$ can be estimated in the same way as in Lemma  \ref{Le10}. Specially, for $k=0$, we have
\[
F_4^0=-\widetilde{\Pi}(v-\widetilde{v})-\tau\left(\mu\frac{\widetilde{u}^R_{\xi}}{\widetilde{v}^R}\right)_t-
    \sigma\tau\left(\mu\frac{\widetilde{u}^R_{\xi}}{\widetilde{v}^R}\right)_{\xi}
  +(\widetilde{v}^R-v_m)(\widetilde{\Pi}^S)^{-X}
  +\left((\widetilde{v}^S)^{-X}-v_m\right)\left(\mu\frac{\widetilde{u}^R_{\xi}}{\widetilde{v}^R}\right),
\]
so, we have
\[
\begin{aligned}
N_5^0\leq &C\delta_S\int_0^t\int_{\mathbb{R}}\left(\partial_{\xi}\Psi\right)^2\dif\xi \dif t+C\delta_S\int_0^t\mathcal{G}^S(U)\dif t+C\delta_R^{\theta}.
\end{aligned}
\]
For $k=1$, the term $F_3^1$ is the same as in Lemma \ref{Le10}. Thus, using similar estimates as in Lemma \ref{Le10}, we have
\[
N_5^1\leq C\delta\int_0^t\|\partial_{\xi}\Phi\|^2_{H^1}\dif t
+C\delta\int_0^tG(U)\dif t
   +C\delta_S\int_0^t\mathcal{G}^S(U)\dif t+C\delta_R^{\theta}.
\]
Therefore, combining the above estimates, we get the desired results.
\end{proof}
Choosing $\epsilon, \epsilon_1$ small enough and combining the results of Lemma \ref{Le5}, Lemma \ref{Le10}, Lemma \ref{Le11}, Lemma \ref{Le12}, the proof of the Proposition \ref{p1} is finished.

\section{Proof of Theorem \ref{th1.2}}
In this section, we show the Theorem \ref{th1.2} by use of the uniform estimates obtained in Sec.4 and usual compactness arguments. 
Firstly, using the definition of error terms ($\Phi=v-\widetilde v, \Psi=u-\widetilde u, Q=\Pi-\widetilde \Pi$) and a priori estimates (Proposition \ref{p1}), we get
\begin{align*}
\sup_{0\le t<\infty}\|(\Phi^{\tau}, \Psi^{\tau}, \sqrt{\tau}Q^{\tau}) (t,\cdot)\|_{H^2}^2+\int_0^{\infty}\left(\|(\Phi^{\tau}_{\xi}, \Psi^{\tau}_{\xi})\|_{H^1}^2+\|Q^{\tau}\|_{H^2}^2\right)\dif t
\le C_0E(0)+C_0\delta_R^{\theta},
\end{align*}
where $\Phi^{\tau}=v^{\tau}-\widetilde v^{\tau}, \Psi^{\tau}=u^{\tau}-\widetilde u^{\tau}, Q^{\tau}=\Pi^{\tau}-\widetilde \Pi^{\tau}$, $\theta=\min\{\frac{1}{2}, \frac{3}{2}-\frac{1}{q}\}$, $C_0$ is a constant independent of $\tau$ and $\widetilde v^{\tau}=(\widetilde v^S)^{\tau}+\widetilde v^R-v_m, \widetilde u^{\tau}=(\widetilde u^S)^{\tau}+\widetilde u^R-u_m, \widetilde \Pi^{\tau}=(\widetilde \Pi^S)^{\tau}+\mu\frac{(\widetilde u^R)_{\xi}}{\widetilde v^R}$ are the compose waves of system \eqref{eq1}, respectively. Thus, there exist $(\Phi^0, \Psi^0)\in L^{\infty}((0,\infty);H^2)$ and $Q^0\in L^2((0, \infty);H^2)$ such that
\begin{equation}\label{5.1}
\begin{aligned}
(\Phi^{\tau}, \Psi^{\tau})\rightharpoonup(\Phi^0, \Psi^0)\qquad weakly-* \quad in \quad L^{\infty}((0,\infty);H^2),\\
Q^{\tau}\rightharpoonup Q^0 \qquad weakly- \quad in \quad  L^2((0, \infty);H^2).
\end{aligned}
\end{equation}
Secondly, from Lemma \ref{Le2}, we get
\[
\begin{aligned}
&\|(\widetilde v^S)^{\tau}-v_+\|_{L^2(\mathbb{R}^+)}+\|(\widetilde v^S)^{\tau}-v_m\|_{L^2(\mathbb{R}^-)}
+\|(\widetilde v^S)^{\tau}_{\xi}\|_{H^2}\le C,\\
&\|(\widetilde u^S)^{\tau}-u_+\|_{L^2(\mathbb{R}^+)}+\|(\widetilde u^S)^{\tau}-u_m\|_{L^2(\mathbb{R}^-)}
+\|(\widetilde u^S)^{\tau}_{\xi}\|_{H^2}\le C,\\
&\|(\widetilde \Pi^S)^{\tau}\|_{H^2}\le C,
\end{aligned}
\]
where $C$ independent of $\tau$.\\
Then, using compactness theorem, we have
\[
\begin{aligned}
(\widetilde v^S)^{\tau}\rightarrow(\widetilde v^S)^0,\quad(\widetilde u^S)^{\tau}\rightarrow(\widetilde u^S)^0,\qquad strongly\quad in \quad H^2,\\
(\widetilde \Pi^S)^{\tau}\rightharpoonup(\widetilde \Pi^S)^0,\qquad weakly-\quad in \quad H^2.
\end{aligned}
\]
In addition, let $\tau\rightarrow0$ in \eqref{eq25}, we have
\[
(\widetilde \Pi^S)^{\tau}\rightharpoonup\mu\frac{(\widetilde u^S)_{\xi}^0}{(\widetilde v^S)^0}:=(\widetilde \Pi^S)^{0},\qquad weakly-\quad in \quad H^2
\]
and we know that $(\widetilde v^S)^0,(\widetilde u^S)^0$ are the traveling wave solutions of classical Navier-Stokes equations.
Therefore, we have
\begin{equation}\label{5.2}
\begin{aligned}
\widetilde v^{\tau}=(\widetilde v^S)^{\tau}+\widetilde v^R-v_m\rightarrow(\widetilde v^S)^{0}+\widetilde v^R-v_m:=\widetilde v^{0}, \quad strongly\quad in \quad C([0, \infty); H^2),\\
 \widetilde u^{\tau}=(\widetilde u^S)^{\tau}+\widetilde u^R-u_m\rightarrow(\widetilde u^S)^{0}+\widetilde u^R-u_m:=\widetilde u^{0}, \quad strongly\quad in \quad C([0, \infty); H^2),\\
  \widetilde \Pi^{\tau}=(\widetilde \Pi^S)^{\tau}+\mu\frac{(\widetilde u^R)_{\xi}}{\widetilde v^R}\rightharpoonup(\widetilde \Pi^S)^{0}+\mu\frac{(\widetilde u^R)_{\xi}}{\widetilde v^R}:=\widetilde \Pi^{0},\qquad weakly- \quad in \quad C([0, \infty); H^2).
\end{aligned}
\end{equation}
Finally, for any $T>0$, using \eqref{50}, we know that $\Phi_t^{\tau}$ and $\Psi_t^{\tau}$ are bounded in $L^2((0,T);H^1)$. Thus, $(\Phi^0, \Psi^0)$ are relatively compact in $C([0,T];H^1)$. Furthermore, using compactness theorem, for any $\alpha>0$, $(\Phi^{\tau},\Psi^{\tau})$ are relatively compact in $C([0,T];H^{2-\alpha})$. Then, as $\tau\rightarrow0$, we have
\[
(\Phi^{\tau},\Psi^{\tau})\rightarrow(\Phi^0, \Psi^0)\qquad strongly\quad in \quad C([0,T];H^{2-\alpha}).
\]
Therefore, combining \eqref{5.2}, we have
\begin{align}\label{5.3}
(v^{\tau},u^{\tau})\rightarrow(\Phi^0+\widetilde v^{0}, \Psi^0+\widetilde u^{0})
:=(v^0, u^0)
\qquad strongly\quad in \quad C([0,T];H^{2-\alpha}).
\end{align}
On the other hand, noting that $\sqrt{\tau}Q^{\tau}$ is uniform bounded, we get $\tau Q^\tau\rightarrow0$ in $L^{\infty}((0,\infty);H^2)$ as $\tau\rightarrow0$, which yields $\tau Q_t ^{\tau}\rightarrow0$ in $D^{\prime}((0,\infty);H^2)$ as $\tau\rightarrow0$. Using Lemma \ref{Le2} and let $\tau\rightarrow0$ in \eqref{eq1} and \eqref{eq25}, we have
\begin{align}\label{5.4}
\Pi^{\tau}\rightharpoonup\mu\frac{(u^0)_x}{v^0}:=\Pi^0\qquad a.e. \quad (0,\infty)\times\mathbb{R}
\end{align}
and we conclude that $v^0,u^0$ are the solutions of classical Navier-Stokes equations.
Then, combining \eqref{5.1}, \eqref{5.2}, \eqref{5.3} and \eqref{5.4}, we get the desired results.

\end{document}